\documentclass[11pt,a4paper,reqno]{amsart}

\usepackage{a4wide}
\usepackage{graphicx}
\usepackage{latexsym,amsmath}
\usepackage{epsfig}
\usepackage{amssymb}
\usepackage{amstext}
\usepackage{amsgen}
\usepackage{amsxtra}
\usepackage{amsgen}
\usepackage{amsthm}
\usepackage{tcolorbox}
\usepackage{mathrsfs,dsfont,mathtools,scalerel,amscd,amsfonts,tikz}

\usepackage{hyperref}

                     \usepackage[backend=biber,
            style=numeric-comp,
            sortcites=true,
            giveninits=true, 
            useprefix=true,
            sorting=nyt,
            maxbibnames=9]{biblatex}
            \renewbibmacro{in:}{}
\DeclareFieldFormat{titlecase}{#1}
\usepackage{bm}

\newtheorem{thm}{Theorem}
\newtheorem{prop}{Proposition}[section]
\newtheorem{lemma}[prop]{Lemma}
\newtheorem{cor}[prop]{Corollary}

\theoremstyle{definition}

\newtheorem{definition}[prop]{Definition}

\DeclarePairedDelimiter\ceil{\lceil}{\rceil}

\theoremstyle{remark}
\newtheorem{rem}[prop]{Remark}

\numberwithin{equation}{section}

\newcommand{\per}{\mathrm{per}}

\newcommand{\RR}{\mathbb{R}}

\DeclareMathOperator{\dv}{div}
\DeclareMathOperator{\esssup}{ess \, sup}

\newcommand{\Pe}{\mathrm{Pe}}

\def\Xint#1{\mathchoice
{\XXint\displaystyle\textstyle{#1}}%
{\XXint\textstyle\scriptstyle{#1}}%
{\XXint\scriptstyle\scriptscriptstyle{#1}}%
{\XXint\scriptscriptstyle\scriptscriptstyle{#1}}%
\!\int}
\def\XXint#1#2#3{{\setbox0=\hbox{$#1{#2#3}{\int}$ }
\vcenter{\hbox{$#2#3$ }}\kern-.6\wd0}}

\def\avint{\Xint-}

\newcommand{\loc}{\mathrm{loc}}

\newcommand{\p}{{\bf p}}

\newcommand{\e}{{\bf e}}
\newcommand{\x}{{ x}}

\newcommand{\bxi}{\boldsymbol{\xi}}
\newcommand{\bzeta}{\boldsymbol{\zeta}}

\renewcommand{\d}{\, \mathrm{d} }
\newcommand{\der}{\mathrm{d}}

\renewcommand{\leq}{\leqslant}
\renewcommand{\geq}{\geqslant}

\def\de{\partial}

\def\dv{\mathrm{div}}
\newcommand{\nm}[1]{\left\| #1 \right\|}

\renewcommand{\rho}{\varrho}

\usepackage{xcolor}

\newcommand{\fd}{{\dot{f}}}
\newcommand{\rd}{{\dot{\rho}}}
\newcommand{\ff}[1]{f^{(#1)}}
\newcommand{\rr}[1]{\rho^{(#1)}}

\def\de{\partial}

\begin{document}

\title[Regularity and trend to equilibrium]{Regularity and trend to equilibrium for a non-local advection-diffusion model of active particles}
\author[L.~C.~B.~Alasio]{Luca Alasio}
\address[L.~C.~B.~Alasio]{INRIA Paris \& Laboratoire Jacques-Louis Lions, Sorbonne Universit\'e, 4 Place Jussieu, 75005 Paris, France}
\email{luca.alasio@sorbonne-universite.fr}

\author[J.~Guerand]{Jessica Guerand}
\address[J.~Guerand]{Institut Montp\'ellierain Alexandre Grothendieck UMR 5149,
Université de Montpellier, Place Eugène Bataillon,
34090 Montpellier, France}
\email{jessica.guerand@umontpellier.fr}

\author[S.~M.~Schulz]{Simon Schulz}
\address[S.~M.~Schulz]{Scuola Normale Superiore, Centro di Ricerca Matematica Ennio De Giorgi, P.zza dei Cavalieri, 3,  56126 Pisa, Italy}\email{simon.schulz@sns.it}



\maketitle

\begin{abstract}

    We establish regularity and, under suitable assumptions, convergence to stationary states for weak solutions of a parabolic equation with a non-linear non-local drift term; this equation was derived from a model of active Brownian particles with repulsive interactions in the previous work \cite{bbesModel}, which incorporates advection-diffusion processes both in  particle position and orientation. We apply De Giorgi's method and differentiate the equation with respect to the time variable iteratively to show that weak solutions become smooth away from the initial time. This strategy requires that we obtain improved integrability estimates in order to cater for the presence of the non-local drift. The instantaneous smoothing effect observed for weak solutions is shown to also hold for very weak solutions arising from distributional initial data; the proof of this result relies on a uniqueness theorem in the style of M.~Pierre for low-regularity solutions. The convergence to stationary states is proved under a smallness assumption on the drift term. 
\end{abstract}

\setcounter{tocdepth}{1}
\tableofcontents

\section{Introduction}\label{sec:intro}

This work is concerned with the study of the regularity and convergence to stationary states for the non-local advection-diffusion equation 
\begin{equation}\label{eq:main eqn}
    \partial_t f + \Pe ~ \dv \big(  (1- \rho) f \e(\theta)\big)
= D_e \Delta f + \partial_{\theta}^2 f, 
\end{equation}
where $\rho(t,x) = \int_0^{2\pi} f(t,x,\theta) \d \theta$ is the \emph{angle-independent density} and $\e(\theta) = (\cos \theta, \sin \theta)$, with periodic boundary conditions both in the space variable $x \in \Omega = (0,2\pi)^2$ and the angle variable $\theta \in (0,2\pi)$; we use the notation $\Upsilon = \Omega \times (0,2\pi) = (0,2\pi)^3$. The operators $\dv$ and $\Delta$ are taken with respect to $x$ only. The constant parameters $\Pe \in \mathbb{R}$ and $D_e > 0$ are called the \emph{P\'eclet number} and \emph{spatial diffusion coefficient}, respectively. This equation was formally derived from a many-particle system in \cite{bbesModel}; the particles are said to be \emph{active} in the sense that they are self-propelled, with velocity pointing in the direction of the vector $\e(\theta)$. 

\subsection*{Context}
The regularity of elliptic and parabolic equations has been a topic of primordial importance within analysis and partial differential equations since the advent of Hilbert's 19th problem; we mention in particular the seminal contributions of De Giorgi \cite{degiorgi1,degiorgi2}, Nash \cite{nash1,nash2}, and Moser \cite{moser1,moser2}. The strategy in the aforementioned works is common: one estimates the $L^p$ norms of the weak solution \emph{locally} for arbitrary large $p$ by an iteration procedure so as to obtain local $L^\infty$ bounds. This interior boundedness is then used to obtain local oscillation estimates from which one infers H\"older continuity of the solution; either by contradiction arguments (\textit{e.g.}~\cite{vasseur}) or by quantitative methods (\textit{e.g.}~\cite{JessParabolic}). These approaches were subsequently refined and generalised to cover certain classes of degenerate elliptic equations by Ladyzhenskaya \emph{et al.}~\cite{ladyzhenskaia1988linear}, and similar degenerate parabolic equations by DiBenedetto \emph{et al.}~\cite{DBdegenerate1,DBdegenerate2,DBdegenerate3,DBdegenerate4,DBdegenerate5}. In recent years, the method has been employed to obtain analogous regularity results for the Navier--Stokes system \cite{vasseurfluid}, systems of reaction-diffusion equations \cite{VasseurCaputo,VasseurGoudon}, equations incorporating fractional diffusion \cite{VasseurCaff1,VasseurCaff2,FelsingerKassmann,KriegerStrain}, as well as for kinetic Fokker--Planck equations \cite{ClementJess,Bouchut,JessCyril,Pascucci,Wendong1,Wendong2}; also known as ultraparabolic or Kolmogorov-type equations. Some aspects of regularity for a class of Fokker--Planck equations without advection were discussed in \cite{jko}, while a more general account of the underlying physics encapsulated by such equations may be found in the monograph \cite{FrankBig}.

The long-time behaviour of solutions to semi-linear advection-diffusion equations of the form 
\begin{equation*}
    \partial_t v(t,{\bf x}) = \Delta_{\bf x} v + \dv_{\bf x} (  v \, {\bf b} )  \qquad 
   ({\bf x}\in\RR^d),
\end{equation*}
which includes \eqref{eq:main eqn}, is also a classical area of interest. Solutions of such equations exhibit a rich panorama of possible properties, depending on the structure of the term ${\bf b}$; some typical examples are as follows. 
\begin{enumerate}
    \item \textit{Linear equation} ${\bf b} = G({\bf x})$: the asymptotic behaviour depends on the existence of a suitable Lyapunov functional. A necessary condition for the existence of stationary states is (\textit{cf}.~\cite{huang2015steady}):
    \begin{equation*}
        \lim_{\bf x \to \infty} G({\bf x}) \cdot {\bf x} > 0.
    \end{equation*}
    This is guaranteed, for example, if $G=\nabla V$, with $V:\RR^d\to\RR$ being a convex function. Indeed exponential convergence to the stationary state may be obtained in this case thanks to the gradient flow approach, the entropy/entropy-dissipation method, or the Bakry--Emery Theory; see \textit{e.g.}~\cite[Ch.~10]{AGS08} and \cite[Ch.~2]{jungel2016entropy}.
    \item \textit{Aggregation-diffusion equation} ${\bf b} = \nabla W*v$: for sufficiently smooth kernels, the solution approaches the heat kernel in $L^1(\RR^d)$ with a polynomial rate (see \cite{carrillo2023asymptotic}). For singular interaction kernels blow-up may occur as in the Keller--Segel case.
    \item \textit{Blow-up and non-existence of stationary states}: examples exist even in the case of one spatial dimension.
    In the super-linear case we mention the case of  ${\bf b} = v^q$, for $q>1$, for which the critical mass was determined in \cite{alikakos1989blow}.
    In the non-local setting we mention the case of ${\bf b} = x + \delta + \int y v \d y$ for which there are no stationary states on $\RR$, see \cite{Bogachev}.
    \item \textit{Space-time periodic solutions:} we refer, for example, to \cite{ji2019existence, barles2001space}. A very interesting example of advection-diffusion equation displaying ``anomalous diffusion'' was recently obtained in \cite{armstrong2023anomalous}, where the vector field {\bf b} is constructed to be space-time periodic, divergence-free, H\"older-continuous and ``fractal''.
    \end{enumerate}
    Furthermore, the structure of the drift term $\mathbf{b}$ has a profound impact on the regularity of the stationary solutions, and it is \emph{a priori} unclear if whether or not such stationary states are smooth with respect to the space variable. 
 
\subsection*{Motivation and originality}
Equation \eqref{eq:main eqn}, derived in \cite{bbesModel} and analysed in \cite{bbes}, is different to the ones studied in the aforementioned works. It incorporates a drift-diffusion mechanism similar to kinetic Fokker--Planck equations, yet the drift term comprises the angle-independent density and is thereby non-local in addition to not being divergence-free. To the authors' knowledge, the regularity and long-time behaviour of equations involving such drift terms has not been addressed in the existing literature, except under restrictive assumptions (see, \textit{e.g.}, \cite{Bogachev}); however, concerning existence and uniqueness for similar models, we highlight the recent work \cite{BriantMeunier}. In particular, it is \emph{a priori} not clear whether solutions are merely H\"older continuous away from the initial time or if they become infinitely-differentiable, nor is it obvious that they converge to stationary states. While some continuity results had been obtained in \cite[\S 4]{bbes} by employing the Duhamel principle and assuming more regularity on the initial data, a thorough local regularity analysis had yet to be performed, which motivates the need for the current paper. Additionally, the aforementioned examples of long-time behaviour illustrate that, to the best of our knowledge, there is no unified approach to the study of convergence to equilibrium for generic advection-diffusion equations. We thereby propose an approach based on \emph{a priori} estimates tailored to the problem at hand and, as a result, obtain strong regularity properties.

The present contribution provides a regularity analysis \emph{\`a la} De Giorgi for \eqref{eq:main eqn}. We first show interior local boundedness of weak subsolutions of \eqref{eq:main eqn} by an iteration procedure, which then enables us to show smoothness of weak solutions away from the initial time by means of a bootstrapping argument; the periodicity in space-angle means that we do not need to restrict ourselves to small subcylinders with respect to these variables, whence our result is global in space-angle. We then extend this higher-regularity result to very weak solutions, which arise from merely distributional initial data; the proof relies on a uniqueness theorem for such low-regularity solutions (see Theorem \ref{thm:weak strong uniqueness}). Throughout the paper, we shall make use of the fact that the angle-independent density $\rho$ (and its derivatives) admits higher integrability than the original unknown $f$; reminiscent of velocity averaging lemmas in kinetic theory (see, \textit{e.g.}, \cite{PerthameKinetic}). Finally, by adapting the aforementioned bootstrapping approach, we show the smoothness of stationary solutions. We also show exponential convergence to these stationary states with respect to the $L^2$ norm under a smallness assumption on the P\'eclet number.

\subsection*{Plan of the paper} The rest of the paper is organised as follows. In \S \ref{sec:setup} we recall the original notions of weak and very weak solutions for the equation \eqref{eq:main eqn} introduced in the previous work \cite{bbes}, and provide the statements of our main theorems in \S \ref{sec:main results}. In \S \ref{sec:weak solutions} we provide an alternative notion of weak solution which is better suited to our regularity analysis as well as the main rescaling lemma used for the De Giorgi method. \S \ref{sec:boundedness of weak sol} is concerned with the boundedness of weak solutions; we show the local-in-time boundedness away from the initial time in \S \ref{sec:first de giorgi lemma} for generic admissible initial data, and also provide a global-in-time estimate in \S \ref{subsec:alikakos} for more regular initial data. We study the higher regularity of weak solutions in \S \ref{sec:higher reg} by means of a bootstrapping argument, which yields smoothness of weak solutions away from the initial time. \S \ref{sec:very weak} generalises the higher regularity results obtained for weak solutions in \S \ref{sec:higher reg} to the very weak solutions, where we also prove a uniqueness result for very weak solutions using an argument \emph{\`a la} Michel Pierre. In \S \ref{sec:harris} we prove the smoothness of non-negative stationary solutions, and also show the convergence to stationary states under the assumption of small P\'eclet number. Appendix \ref{sec:appendix} is devoted to the proofs of some technical lemmas and is divided into two parts: Appendix \ref{app:CZ periodic} contains the proof of a version of a Calder\'on--Zygmund Theorem applicable to our periodic setting; Appendix \ref{app:alternative weak formulation} contains the proof of the alternative weak formulation introduced in \S \ref{sec:weak solutions}.

\subsection*{Notations and functional setting} Throughout this work, we use the shorthand $\bxi = (\x,\theta)$ for the concatenated space-angle variable. The letter $C$ will always denote a positive constant independent of $t,\bxi$, unless explicitly stated otherwise, and may change from line to line. For a point $(t_0,\bxi_0)$, we define the \emph{open parabolic cylinder} of radius $r$ by 
\begin{equation}\label{eq:cylinder def}
    Q_r(t_0,\bxi_0) := \{ (t,\bxi) : -r^2 < t-t_0 \leq 0, \, | \bxi-\bxi_0| < r \}, 
\end{equation}
and we shall write $Q_r(0) = Q_r$ when $(t_0,\bxi_0) = 0$. The domain on which the equation is posed is denoted by $\Upsilon_T := (0,T)\times\Upsilon$, and similarly we write $\Omega_T = (0,T)\times\Omega$. We also define the usual parabolic norm 
\begin{equation}\label{eq:parabolic norm}
    \Vert f \Vert_{\mathscr{P}}^2 := \Vert f \Vert_{L^\infty(0,T;L^2(\Upsilon))}^2 + \Vert \nabla_{\bxi} f \Vert_{L^2(\Upsilon_T)}^2, 
\end{equation}
and the corresponding space $\mathscr{P}$ of functions $f:\Upsilon_T \to \mathbb{R}$ with $\Vert f \Vert_\mathscr{P}$ finite. Unless stated explicitly otherwise, the symbols $\dv,\nabla,\Delta$ denote divergences, gradients, and Laplacians taken with respect to the space variable $x$, while the symbols $\dv_{\bxi},\nabla_{\bxi},\Delta_{\bxi}$ denote operators with respect to the concatenated space-angle variable $\bxi=(x,\theta)$.

The topological dual of a function space $E$ is denoted by $E'$, and the bracket $\langle \cdot, \cdot \rangle$  denotes the dual pairing between elements of a space and its dual. In what follows, the functions spaces denoted by $Z_\per(S)$ for $S \in \{(0,2\pi)^d\}_{d=1}^3$, \textit{i.e.}~$d=2$ corresponding to $\Omega$ and $d=3$ to $\Upsilon$, with $Z \in \{L^p,W^{k,p},C^k\}$, are understood to mean 
\begin{equation*}
    Z_\per(S) := \big\{ g:\mathbb{R}^d \to \mathbb{R}: \, \Vert g \Vert_{Z(S)} < \infty, \text{ and } g(y+2\pi \e_i) = g(y) ~ \forall y \in \mathbb{R}^d, i \in \{1,\dots,d\}  \big\}, 
\end{equation*}
where $\{\e_i\}_{i=1}^d$ is the standard basis of $\mathbb{R}^d$. We also refer to such functions as being \emph{$S$-periodic}; by which we mean that such functions are periodic with periodic cell $S$. We denote the spaces 
\begin{equation*}
    \begin{aligned}
                &Y := H^1_\per(\Upsilon) \cap L^2_\per(\Omega;H^2_\per(0,2\pi)), \quad X:= L^2(0,T;Y).  
    \end{aligned}
\end{equation*}
 We denote the space-angle average of a function $f$ by 
\begin{equation*}
    \langle f \rangle := \avint_\Upsilon f \d \bxi = \frac{1}{|\Upsilon|}\int_\Upsilon f \d \bxi. 
\end{equation*}

\subsection{Definitions and problem set-up}\label{sec:setup}
For the purposes of the regularity analysis, the coefficients $\Pe$ and $D_e$ appearing in \eqref{eq:main eqn} do not matter; the following paragraph is concerned with rescaling the variables so as to make them both vanish from the equation. By defining $(a,b,c) := (D_e \Pe^{-2},D_e \Pe^{-1},\sqrt{D_e}\Pe^{-1})$, we see that the rescaled functions 
\begin{equation*}
    \begin{aligned}
        \tilde{f}(t,\x,\theta) := f(at,b\x,c\theta), \quad  \tilde{\rho}(t,\x) := c\int_{0}^{\frac{2\pi}{c}} \tilde{f}(t,\x,\theta) \d \theta = \rho(at,b\x), \quad   \tilde{\e}(\theta) := \e(c\theta), 
    \end{aligned}
\end{equation*}
satisfy, on the rescaled domain $(0,a^{-1}T)\times b^{-1}\Omega \times (0,c^{-1}2\pi)$, the equation 
\begin{equation*}
   \begin{aligned}
       \partial_t \tilde{f} + \dv\big( (1-\tilde{\rho})\tilde{f} \tilde{\e}(\theta) \big) = \Delta \tilde{f} + \partial^2_\theta \tilde{f}, 
   \end{aligned} 
\end{equation*}
whence the constants $\Pe,D_e$ have vanished. The above rescaling alters the periodicity of the functions $\tilde{f}$, $\tilde{\rho}$ and $\tilde{\e}(\theta)$. This is of no importance whatsoever, since the proofs of the well-posedness results of \cite{bbes} extend easily to any choice of space-angle period. Moreover, the period with respect to the variable $\x$ need not be equal to that with respect to $\theta$. 

Therefore, for the sake of clarity and concision, the regularity analysis in this work is concerned with the study of the drift-diffusion equation 
\begin{equation}\label{eq:main eqn intro}
    \partial_t f + \dv_{\bxi}(Uf) = \Delta_{\bxi} f \qquad \text{in } \Upsilon_T, 
\end{equation}
where $\bxi =(\x,\theta)$ is the concatenated space-angle variable and \begin{equation}\label{eq:U}
    U = (1-\rho)\left(\begin{matrix}\e(\theta) \\ 
0 \end{matrix}\right).
\end{equation}

For convenience, we recall the definitions of weak and very weak solutions introduced in \cite{bbes}, as well as the well-posedness results proved therein.

\begin{definition}[Notions of Solution]\label{def:concept of solution}\quad

We introduce the notions of \emph{weak solution} and \emph{very weak solution}. 

\begin{enumerate}
    \item A \emph{weak solution} of \eqref{eq:main eqn} with non-negative initial data $f_0 \in L^2_\per(\Upsilon)$ satisfying 
    \begin{equation}\label{eq:initial data assumptions}
    \rho_0(x) = \int_0^{2\pi} f_0(x,\theta) \d \theta \in [0,1] \text{ a.e.~}x\in\Omega   
    \end{equation}
    is a function $f\in C([0,T];L^2_\per(\Upsilon)) \cap L^2(0,T;H^1_\per(\Upsilon))$ with 
$\partial_t f\in L^2(0,T;(H^1_\per)'(\Upsilon))$ such that, for all $\varphi \in L^2(0,T;H^1_\per(\Upsilon))$, there holds 
    \begin{equation}\label{eq:weak_form_def}
  \left\lbrace  \begin{aligned}
   & \left\langle\partial_t f, \varphi\right\rangle = \Pe \int_{\Upsilon_T}\!\!\! (1-\rho) f\e(\theta) \cdot \nabla\varphi \d \bxi \d t - D_e \int_{\Upsilon_T}\!\!\! \nabla f \cdot \nabla \varphi \d \bxi \d t - \int_{\Upsilon_T} \!\!\! \partial_\theta f \cdot \partial_\theta \varphi \d \bxi \d t,\\
 &\lim_{t \to 0^+}f(t)=f_0 \text{ in } L^2_{\per}(\Upsilon),
\end{aligned}\right.
    \end{equation}
where $\rho(t,x)=\int_0^{2\pi}f(t,x,\theta) \d \theta$ satisfies 
\begin{equation}\label{eq:rho bounds from bbes}
    0 \leq \rho(t,x) \leq 1 \text{ a.e.~}(t,x) \in \Omega_T. 
\end{equation}

\item A \emph{very weak solution} of \eqref{eq:main eqn} with non-negative initial data $f_0 \in L^2_\per(\Omega;(H^1_\per)'(0,2\pi))$ satisfying 
\begin{equation}\label{eq:initial data assumptions very weak}
    \rho_0(x) = \langle f_0(x,\cdot),1 \rangle \in [0,1] \text{ a.e.~}x\in\Omega 
\end{equation}
is a function $$f \in L^2(0,T;H^1_{\per}(\Omega;(H^1_{\per})'(0,2\pi))) \cap L^\infty(0,T;L^2_{\per}(\Omega;(H^1_{\per})'(0,2\pi))) \cap L^2(\Upsilon_T)$$ with $\partial_t f \in X'$ such that, for all $\varphi \in X$, there holds 
    \begin{equation}\label{eq:weak_form_def_distributional}
    \left\lbrace\begin{aligned}
   & \left\langle\partial_t f, \varphi\right\rangle = \Pe \int_{\Upsilon_T}\!\!\! (1-\rho) f \e (\theta) \cdot \nabla \varphi \d \bxi \d t - D_e \int_{\Upsilon_T} \!\!\! \nabla f \cdot \nabla \varphi \d \bxi \d t + \int_{\Upsilon_T} \!\!\! f \partial^2_\theta \varphi \d \bxi \d t,\\
  &  \lim_{t \to 0^+}f(t) = f_0 \text{ in }Y', 
\end{aligned}\right.
    \end{equation}
where $\rho(t,x)=\int_0^{2\pi}f(t,x,\theta) \d \theta$ satisfies the estimate \eqref{eq:rho bounds from bbes}. 
\end{enumerate}
\end{definition}

We recall that it was proved in \cite{bbes} that, if the initial data $f_0 \in L^2_\per(\Upsilon)$ is non-negative and satisfies \eqref{eq:initial data assumptions}, then there exists a unique weak solution of \eqref{eq:main eqn}. Similarly, if the initial data $f_0 \in L^2_\per(\Omega;(H^1_\per)'(0,2\pi))$ is non-negative and satisfies \eqref{eq:initial data assumptions very weak}, then there exists a very weak solution of \eqref{eq:main eqn}. Furthermore, these solutions are \emph{global-in-time} in the sense that they exist on the time interval $(0,T)$ for all $T>0$.

\subsection{Main theorems}\label{sec:main results}

We state our main results, which are essentially partitioned three two: regularity of the time-dependent solutions, regularity of stationary solutions, and convergence to stationary states. For the regularity, as already mentioned, we first prove the smoothness of weak solutions away from the initial time, and then extend this result to the very weak solutions. 

Our first main theorem is the following.

\begin{thm}[Smoothness away from Initial Time]\label{thm:smooth away from initial}
  Assume $f_0$ is non-negative and satisfies \eqref{eq:initial data assumptions}, $T>0$, and let $f$ be the unique weak solution of \eqref{eq:main eqn} with initial data $f_0$. Then, for a.e.~$t \in (0,T)$, there holds $f \in C^\infty((t,T)\times\mathbb{R}^3)$. 
\end{thm}

We then establish that very weak solutions coincide with weak solutions away from the initial time. 

\begin{thm}[Uniqueness for Very Weak Solutions]\label{thm:weak strong uniqueness}
    For any $T>0$, let $f$ be a very weak solution of \eqref{eq:main eqn}. Then, for a.e.~$t \in (0,T)$, $f$ coincides with the unique weak solution of \eqref{eq:main eqn} on the interval $(t,T)$ with initial data $f(t,\cdot)$. 
\end{thm}

From Theorems \ref{thm:smooth away from initial} and \ref{thm:weak strong uniqueness} we obtain the regularity result for very weak solutions.

\begin{thm}[Regularity for Very Weak Solutions]\label{thm:reg very weak}
Assume $f_0$ is non-negative and satisfies \eqref{eq:initial data assumptions very weak}, $T>0$, and let $f$ be a very weak solution of \eqref{eq:main eqn} with initial data $f_0$. Then, for a.e.~$t \in (0,T)$, there holds $f \in C^\infty((t,T)\times\mathbb{R}^3)$. 
\end{thm}

We also record the following global-in-time boundedness for initial data in $L^\infty$, based on an iterative argument which had been used in the context of degenerate diffusion equations with drift (\textit{cf.}~\cite{KimZhangII}); this result is proved in \S \ref{subsec:alikakos}.

\begin{thm}[Global-in-time Boundedness for Bounded Initial Data]\label{thm:alikakos}
Assume $f_0 \in L^\infty(\Upsilon)$ is non-negative and satisfies \eqref{eq:initial data assumptions}, and let $f$ be the unique weak solution of \eqref{eq:main eqn} with initial data $f_0$. Then there holds the global-in-time estimate 
    \begin{equation*}
        \Vert f \Vert_{L^\infty((0,\infty)\times\Upsilon)} \leq C(\Pe,D_e,\Vert f_0 \Vert_{L^\infty(\Upsilon)}). 
    \end{equation*}
\end{thm}

The final section of this manuscript is concerned with the long-time behaviour of solutions of equation \eqref{eq:main eqn}. Our main results are concerned with the regularity of solutions of the stationary elliptic problem 
\begin{equation}\label{eq:equilibrium eqn}
    \Pe \, \dv\big( (1-\rho_\infty)f_\infty \e(\theta) \big) = D_e \Delta f_\infty + \partial^2_\theta f_\infty, 
\end{equation}
and the convergence of the time-dependent solutions to such stationary solutions.

 \begin{thm}[Smoothness of Stationary States]\label{thm:smooth stationary states}
     Let $f_\infty \in H^1_\per(\Upsilon)$ with $\rho_\infty = \int_0^{2\pi} f_\infty \d \theta \in [0,1]$ be a non-negative periodic weak solution of \eqref{eq:equilibrium eqn}, \textit{i.e.}~for all $\phi \in H^1_\per(\Upsilon)$ there holds 
     \begin{equation*}
        \Pe \int_\Upsilon (1-\rho_\infty)f_\infty \e(\theta) \cdot \nabla \phi \d \bxi = D_e \int_\Upsilon \nabla_{\bxi} f_\infty \cdot \nabla_{\bxi} \phi \d \bxi. 
     \end{equation*}
     Then, $f_\infty$ is a smooth periodic function on $\mathbb{R}^3$. 
 \end{thm}   

\begin{thm}[Convergence to Constant Stationary States for Small P\'eclet Number]\label{thm:small peclet conv}
    Assume $f_0 \in L^2_\per(\Upsilon)$ is non-negative and satisfies \eqref{eq:initial data assumptions}, and let $f$ be the unique solution of \eqref{eq:main eqn} with initial data $f_0$. Assume that there holds $$|\Pe| < \frac{\min\{D_e,1\}}{2\sqrt{2}\pi C_P(1+\langle f_0 \rangle)},$$ 
where    $C_P$ is the Poincar\'e constant associated to $\Upsilon$, and define 
\begin{equation}\label{eq:kappa def}
    \kappa := \frac{1}{2}\Big(\frac{1}{2}C_P^{-2}\min\{D_e,1\} - \frac{(2\pi)^2\Pe^2(1+\langle f_0\rangle)^2}{\min\{D_e,1\}}\Big)>0. 
\end{equation}
Then, 
    \begin{equation*}
        \Vert f(t,\cdot) - \langle f_0 \rangle \Vert_{L^2(\Upsilon)} \leq e^{-\kappa t}\Vert f_0 - \langle f_0 \rangle \Vert_{L^2(\Upsilon)} \quad \text{for all } t \geq 0. 
    \end{equation*}
\end{thm}

 \begin{rem}[Stationary States]
     We remark that Theorem \ref{thm:small peclet conv} shows that, under the aforementioned assumptions on the initial data and on the P\'eclet number, all weak solutions of \eqref{eq:main eqn} converge to a constant stationary state in the limit of infinite time; by Theorem \ref{thm:weak strong uniqueness} this is also the case for all very weak solutions. We note that this result is in accordance with the linear stability analysis performed in \cite[\S 3]{bbesModel}. For large P\'eclet number, the simulations in \cite[\S 4]{bbesModel} suggest that phase separation occurs, and we do not expect convergence to a constant stationary state. The study of the long-time behaviour of the solutions for large P\'eclet number will be the subject of future investigations. 
 \end{rem}

In this work, we shall frequently employ a well-known interpolation inequality; we refer to it throughout the paper as the \emph{Interpolation Lemma}. 

\begin{lemma}[DiBenedetto, Proposition 3.2 of \S 1, \cite{DiBenedetto}]\label{lem:dibenedetto classic}
    Let $d \in \mathbb{N}$ and $\omega \subset \mathbb{R}^d$ have piecewise smooth boundary, and let $p,m \geq 1$. There exists a positive constant $C$ depending only on $d,p,m$ and the structure of $\partial \omega$ such that, for all $$v \in L^\infty(0,T;L^m(\omega)) \cap L^p(0,T;W^{1,p}(\omega)) =: V^{m,p},$$ there holds 
    \begin{equation*}
        \Vert v \Vert_{L^q((0,T)\times\omega)} \leq C \bigg( 1 + \frac{T}{|\omega|^{\frac{d(p-m)+mp}{md}}} \bigg)^{\frac{1}{q}} \Vert v \Vert_{V^{m,p}}, \qquad q = p\frac{d+m}{d}. 
    \end{equation*}
\end{lemma}

We shall also employ the following version of the Calder\'on--Zygmund Theorem applicable to our periodic setting. The proof is delayed to Appendix \ref{app:CZ periodic}. We do not claim that this result is sharp nor that it is novel; it is merely sufficient for our purposes. 

\begin{lemma}[Periodic Calder\'on--Zygmund Inequality]\label{lem:CZ periodic}
   Let $p \in (1,\infty)$. There exists a positive constant $C=C(p,\Upsilon)$ such that for all $v \in W^{1,p}_\per(\Upsilon)$ with $\Delta_{\bxi} v \in L^p(\Upsilon)$, there holds 
    \begin{equation*}
        \Vert \nabla^2_{\bxi} v \Vert_{L^p(\Upsilon)} \leq C\Big( \Vert \Delta_{\bxi} v \Vert_{L^p(\Upsilon)} + \Vert v \Vert_{W^{1,p}(\Upsilon)} \Big). 
    \end{equation*}
    In the case $p=2$, there holds 
     \begin{equation*}
        \Vert \nabla^2_{\bxi} v \Vert_{L^2(\Upsilon)} = \Vert \Delta_{\bxi} v \Vert_{L^2(\Upsilon)}. 
    \end{equation*}
\end{lemma}

To conclude this subsection, we briefly outline the mechanism by which the equation improves the regularity of its solution. 

\begin{rem}[Regularity Bootstrap]
    One begins by applying De Giorgi's method/Moser's iteration technique to \eqref{eq:main eqn intro} and obtain that $f \in L^\infty((t,T)\times\Upsilon)$ for a.e.~$t>0$. This boundedness is then sufficient to perform the classical $H^2$-type estimate on \eqref{eq:main eqn intro}; yielding boundedness in $L^2((t,T)\times\Upsilon)$ both for $\Delta_{\bxi} f$ and $\partial_t f =: \dot{f}$. Further work shows that $\dot{f} \in L^\infty(t,T;L^2(\Upsilon)) \cap L^2(t,T;H^1(\Upsilon))$ satisfies 
    \begin{equation}\label{eq:time deriv eqn in intro}
        \partial_t \dot{f} + \dv\big( (\dot{f}(1-\rho) - f \dot{\rho}) \e(\theta) \big) = \Delta_{\bxi} \dot{f}; 
    \end{equation}
   and, \emph{a priori}, the boundedness of $\dot{\rho}$ in $L^2((t,T)\times\Upsilon)$ is insufficient to apply De Giorgi's method to the above and obtain boundedness of $\dot{f}$ in $L^\infty((t,T)\times\Upsilon)$. Indeed, it appears at first glance from the formula $\dot{\rho} = \int_0^{2\pi} \dot{f} \d \theta$ that $\dot{\rho}$ inherits the same boundedness properties as $\dot{f}$ and none more. However, the dimensionality reduction in the angle-independent density plays a crucial role. Using the Interpolation Lemma \ref{lem:dibenedetto classic}, one obtains $\dot{f}  \in L^{10/3}(\Upsilon_T)$, as $\Upsilon \subset \mathbb{R}^3$; see Proof of Lemma \ref{lem:eqn for f dot}. Meanwhile, using this very same interpolation result, since $\Omega \subset \mathbb{R}^2$, we obtain that $\dot{\rho} \in L^\infty(0,T;L^2(\Omega)) \cap L^2(0,T;H^1(\Omega))$ in fact belongs to $L^{4}(\Omega_T)$. This improved integrability in $\dot{\rho}$ leads to classical $H^2$-type estimate on the equation 
\begin{equation*}
    \partial_t \dot{\rho} + \dv\big( \p \dot{\rho} + (1-\rho) \dot{\p} \big) = \Delta \dot{\rho}, 
    \quad \text{ where } \quad \p(t,x) = \int_0^{2\pi} f(t,x,\theta) \e(\theta) \d \theta,
\end{equation*}
    yielding higher integrability of $\dot{\rho}$ again by interpolation. We deduce a sufficient gain of integrability in the non-local drift term of \eqref{eq:time deriv eqn in intro} to then apply De Giorgi's method again and obtain boundedness in $L^\infty$ of the time derivative $\dot{f}$. This procedure is then performed iteratively for all time-derivatives $\partial^n_t f$, from which we deduce smoothness of the solution away from the initial time. 
\end{rem}

\section{Preliminary Notions}\label{sec:weak solutions}

In this section, we present the main rescaling lemma for De Giorgi's method, for which we introduce an alternative weak formulation of the equation; we explain our reasons for doing so in the paragraphs that follow. 

The strategy of our regularity analysis is to use the method of De Giorgi to obtain interior regularity, which involves ``zooming in'' on subcylinders to obtain local boundedness. Central to this strategy is the appropriate parabolic rescaling of the functions at hand to deduce the required \emph{localised} estimates by first obtaining analogous bounds on the unit cylinder $Q_1$. 

It is apparent that the aforementioned rescaling affects the periodicity of the test functions that can be inserted into the weak formulation of Definition \ref{def:concept of solution}. In turn, it is more convenient to employ an alternative weak formulation, for which we do not require the test functions to be periodic; this is encapsulated in the following lemma, the proof of which is delayed to Appendix \ref{app:alternative weak formulation}.

\begin{lemma}[Alternative Weak Formulation]\label{lem:equiv weak}
    Assume $f_0 \in L^2_\per(\Upsilon)$ is non-negative and satisfies \eqref{eq:initial data assumptions}, and let $f \in C([0,T];L^2_\per(\Upsilon)) \cap L^2(0,T;H^1_\per(\Upsilon))$ with $\partial_t f \in L^2(0,T;(H^1_\per)'(\Upsilon))$ be the unique weak solution of \eqref{eq:main eqn} with initial data $f_0$. Then, for all $\varphi \in C^\infty([0,T]\times\mathbb{R}^3)$ with $\varphi(t,\cdot) \in C^\infty_c(\mathbb{R}^3)$ for all $t$ and for a.e.~$t_1,t_2 \in [0,T]$, there holds 
    \begin{equation}\label{eq:equiv weak form}
       \begin{aligned}
            \int_{t_1}^{t_2} \int_{\mathbb{R}^3} f \partial_t \varphi \d \bxi \d t + \int_{t_1}^{t_2} \int_{\mathbb{R}^3} (1-\rho) f \e(\theta) \cdot \nabla \varphi \d \bxi \d t &- \int_{t_1}^{t_2} \int_{\mathbb{R}^3} \nabla_{\bxi} f \cdot \nabla_{\bxi} \varphi \d \bxi \d t \\ &= \int_{\mathbb{R}^3} f \varphi \d \bxi \Big|_{t_2}  -  \int_{\mathbb{R}^3} f \varphi \d \bxi \Big|_{t_1}. 
        \end{aligned}
    \end{equation}
\end{lemma}

With this alternative weak formulation at our disposal, we define our notion of weak subsolution for the generalisation 
\begin{equation}\label{eq:rescaled eqn with V}
    \partial_t f + \dv_{\bxi} (Uf + V) = \Delta_{\bxi}f 
\end{equation}
of the equation \eqref{eq:main eqn intro}, where $V \in L^q(\Upsilon_T)$ for $q>5$, which is the form which we shall use when employing De Giorgi's method. The reason as to why we generalise the analysis to include the term $V$ will be made clear in \S \ref{sec:higher reg}. 

\begin{definition}[Weak Subsolution]\label{def:weak subsol}
We say that $f \in C([0,T];L^2_\loc(\mathbb{R}^3)) \cap L^2(0,T;H^1_\loc(\mathbb{R}^3))$ with $\partial_t f \in L^2(0,T;(H^1_\loc)'(\mathbb{R}^3))$ is a \emph{weak subsolution} of \eqref{eq:rescaled eqn with V} if, for all non-negative $\varphi \in 
C^\infty(0,T;C^\infty_c(\mathbb{R}^3))$ and for a.e.~$t_1,t_2 \in [0,T]$, there holds 
    \begin{equation}\label{eq:rescaled eqn int form}
    \begin{aligned}    -\int_{t_1}^{t_2}\int_{B_1} f \partial_t \varphi \d \bxi & \d t + \int_{t_1}^{t_2}\int_{B_1} \nabla_{\bxi} f \cdot \nabla_{\bxi} \varphi \d \bxi \d t \\ 
    &\leq \int_{t_1}^{t_2}\int_{B_1} (f U + V) \cdot \nabla_{\bxi} \varphi \d \bxi \d t - \int_{B_1} f \varphi \d \bxi \Big|_{t_2}  +  \int_{B_1} f \varphi \d \bxi \Big|_{t_1} . 
    \end{aligned}
    \end{equation}
\end{definition}

\begin{rem}[Parabolic Norm]
We recall that, given admissible initial data $f_0$, it was shown in \cite[Proof of Theorem 3.1, (3.15)]{bbes} that the weak solution $f$ admits, for some positive constant $C$ independent of $T$, the estimate 
\begin{equation}\label{eq:estimation_parabolic}
    \Vert f \Vert^2_{\mathscr{P}} \leq \frac{C}{\pi \, \min \{1,D_e\}} e^{2CT\frac{\Pe^2}{D_e }}\Vert f_0 \Vert^2_{L^2(\Upsilon)}, 
\end{equation}
where the parabolic norm was defined in \eqref{eq:parabolic norm}. 
\end{rem}

We now provide the main rescaling lemma used in the proof of the first De Giorgi lemma, using the notion of subsolution given in the previous definition.

\begin{lemma}[Rescaling Lemma]\label{lem:rescaling i}
    Let $f$ be a weak subsolution of \eqref{eq:main eqn intro}. Let $\delta \in (0,1)$, $(t_0, \bxi_0) \in (0,T)\times\mathbb{R}^3$, and $r$ satisfy the constraint: 
    \begin{equation*}
        0 < r < \min\big\{1,\sqrt{t_0/2}\big\}. 
    \end{equation*}
    Let $(t,\bxi) \in Q_r(t_0,\bxi_0)$ and define 
    \begin{equation*}
        \ell(r,\delta) := \delta^{\frac{1}{2}} \frac{r^{\frac{3}{2}}}{\Vert f \Vert_\mathscr{P} + \Vert V \Vert_{L^q(\Upsilon_T)}}, 
    \end{equation*}
    as well as the rescaled functions $f_r,U_r,V_r:Q_1 \to \mathbb{R}$ by 
     \begin{equation}\label{eq:rescaling i}
        \begin{aligned}
            &f_r(\tau,\bzeta) := \ell f(t+r^2\tau,\bxi + r\bzeta), \\ 
            &U_r(\tau,\bzeta) := rU(t+r^2 \tau, \bxi + r \bzeta), \\ 
            &V_r(\tau,\bzeta) := r\ell V(t+r^2 \tau, \bxi + r \bzeta). 
        \end{aligned}
    \end{equation}
    
Then, $f_r \in C([-1,0];L^2(B_1))\cap L^2(-1,0;H^1(B_1))$, $\partial_t f_r \in L^2(-1,0;(H^1)'(B_1))$, $$\begin{aligned} &|U_r| \leq 
1 \text{ a.e.~} Q_1, \quad \Vert V_r \Vert_{L^q(Q_1)} \leq 1, \quad \esssup_{\tau \in [-1,0]}\int_{B_1}|f_r(\tau)|^2 \d \bzeta + \int_{Q_1} |\nabla_{\bzeta} f_r|^2 \d\bzeta \d \tau \leq \delta, \end{aligned}$$ and $f_r$ is a weak subsolution of 
    \begin{equation}\label{eq:rescaled eqn not rly}
           \partial_\tau f_r + \dv_{\bzeta}( U_r f_r + V_r) = \Delta_{\bzeta} f_r \qquad \text{in } Q_1, 
       \end{equation}
       \textit{i.e.}, for all non-negative $\varphi \in C^\infty(-1,0;C^\infty_c(B_1))$ and for a.e.~$\tau_1,\tau_2 \in [-1,0]$, there holds 
           \begin{equation}\label{eq:rescaled eqn not rly int form}
    \begin{aligned}    -\int_{\tau_1}^{\tau_2}\int_{B_1} f_r \partial_t & \varphi \d \bzeta \d \tau + \int_{\tau_1}^{\tau_2}\int_{B_1} \nabla_{\bzeta} f_r \cdot \nabla_{\bzeta} \varphi \d \bzeta \d \tau \\ &\leq \int_{\tau_1}^{\tau_2}\int_{B_1} (f_r U_r + V_r) \cdot \nabla_{\bzeta} \varphi \d \bzeta \d \tau - \int_{B_1} f_r \varphi \d \bxi \Big|_{\tau_2}  +  \int_{B_1} f_r \varphi \d \bxi \Big|_{\tau_1} . 
    \end{aligned}
    \end{equation}

\end{lemma}

\begin{proof}
The smallness of $r$ and the boundedness of $\rho$ immediately yield the pointwise estimate on $U_r$. Next, observe that 
    \begin{equation}\label{eq:rescaling first formula}
       \begin{aligned}
           \int_{B_1} |f_r(\tau)|^2 \d\bzeta  &\leq \frac{\delta}{\Vert f \Vert_{\mathscr{P}}^2} \int_{B_r} |f(t+r^2 \tau, \bxi+ \bzeta')|^2 \d \bzeta' . 
       \end{aligned} 
       \end{equation}
Note from the definition of the cylinders that 
$$- 2r^2 \leq -r^2 + (t-t_0) \leq t+r^2\tau -t_0 \leq 0,$$ whence 
the conditions on $r$ imply 
\begin{equation}\label{eq:smallness of r i}0 < t+ r^2 \tau <T.
\end{equation}
Furthermore, the smallness of $r$ implies 
\begin{equation}\label{eq:smallness of r ii}
\{\bxi+\bzeta' : \bxi \in B_r(\bxi_0), \bzeta' \in B_r\} \subset \{\bzeta_0+\bzeta: \bzeta \in (-\pi,\pi)^3\}.
\end{equation} 
It follows from \eqref{eq:smallness of r i} and \eqref{eq:smallness of r ii} that there is no overlap in the integration in $\bzeta'$ from one periodic cell to another when performing the integration in \eqref{eq:rescaling first formula}, and thus 
\begin{equation*}
       \begin{aligned}
           \int_{B_r} |f(t+r^2\tau, \bxi+ \bzeta')|^2 \d \bzeta' \leq \Vert f \Vert^2_{L^\infty(0,T;L^2(\Upsilon))}. 
       \end{aligned} 
       \end{equation*}
       Similarly, 
       \begin{equation*}
           \int_{Q_1} |\nabla_{\bzeta}f_r|^2 \d\bzeta \d \tau \leq \frac{\delta}{\Vert f \Vert^2_\mathscr{P}}\int_{Q_r} |\nabla_{\bxi}f(t+\tau',\bxi+\bzeta')|^2 \d \bzeta' \d \tau', 
       \end{equation*}
       and by applying the same reasoning as before, we obtain 
       \begin{equation*}
           \int_{Q_r} |\nabla_{\bxi}f(t+\tau',\bxi+\bzeta')|^2 \d \bzeta' \d \tau' \leq \Vert \nabla_{\bxi} f \Vert^2_{L^2(\Upsilon_T)}, 
       \end{equation*}
       and we deduce 
       \begin{equation*}
           \esssup_{\tau \in [-1,0]}\int_{B_1}|f_r(\tau)|^2 \d \bzeta + \int_{Q_1} |\nabla_{\bzeta} f_r|^2 \d\bzeta \d \tau \leq \delta, 
       \end{equation*}
      as required. Furthermore, 
      \begin{equation*}
          \Vert V_r \Vert^q_{L^q(Q_1)} \leq \frac{r^{5 (\frac{q}{2}-1)}  \delta^{\frac{q}{2}} }{\Vert V \Vert_{L^q(\Upsilon_T)}^q}\int_{Q_r} |V(t' , \bxi') |^q \d \bxi' \d t' \leq 1. 
      \end{equation*}
      The weak subsolution formulation \eqref{eq:rescaled eqn not rly int form} of the drift-diffusion equation \eqref{eq:rescaled eqn not rly} is easily verified by direct computation. 
\end{proof}

It will therefore suffice to study the following equation: 
\begin{equation}\label{eq:rescaled eqn}
    \partial_t f + \dv_{\bxi} (Uf + V) = \Delta_{\bxi}f \qquad \text{in } Q_1, 
\end{equation}
with $f \in C([-1,0];L^2(B_1))\cap L^2(-1,0;H^1(B_1))$, with $\partial_t f \in L^2(-1,0;(H^1)'(B_1))$, $| U| \leq 1$ a.e.~in $Q_1$, and $\Vert V \Vert_{L^{q}(Q_1)} \leq 1$ where $q>5$.

\section{Boundedness of Weak Solutions}\label{sec:boundedness of weak sol}

\subsection{Boundedness away from initial time}\label{sec:first de giorgi lemma}

The goal of this section is to prove the following proposition, which will subsequently be used to prove Theorem \ref{thm:smooth away from initial} in \S \ref{sec:higher reg}.

\begin{prop}[Boundedness away from Initial Time]\label{thm:global boundedness away from initial time}
   Assume $f_0$ is non-negative and satisfies \eqref{eq:initial data assumptions}, $T>0$, and let $f$ be the unique weak solution of \eqref{eq:main eqn} with initial data $f_0$. There exists a positive constant $C$ depending only on $\Upsilon,T$ such that, for all $t \in (0,T)$, 
    \begin{equation}\label{eq:boundedness est loc ii}
        \Vert f \Vert_{L^\infty((t,T)\times\Upsilon)} \leq C (1+t^{-\frac{13}{4}}) \Vert f \Vert_\mathscr{P}. 
    \end{equation}
\end{prop}

Notice that the right-hand side of inequality \eqref{eq:boundedness est loc ii} is bounded thanks to \eqref{eq:estimation_parabolic}.
The proof is broken down into several steps, which constitute the subsections that follow. The local boundedness for solutions of equations of the form \eqref{eq:rescaled eqn} by means of De Giorgi's method is classical (\textit{cf.}~\textit{e.g.}~\cite[Ch.~VI \S 5]{Lieberman}). We nevertheless highlight that a novel aspect of our approach is that we need only consider the solution away from the initial time, and not on a more restrictive subcylinder; this is a consequence of the choice of periodic boundary conditions in space-angle. Details of the iterative procedure for more general systems may also be found in, \textit{e.g.}, \cite[\S 3.2]{vasseur}. We include these details in the present section so as to make the paper self-contained, and to make the proof of the higher-regularity result of \S \ref{sec:higher reg} easier to follow; this latter proof uses De Giorgi's method inductively on repeated time-derivatives of the equation.

\subsubsection{Caccioppoli inequality}

\begin{lemma}[Caccioppoli Inequality]\label{lem:energy ineq weak sol prior to local boundedness}
  Let $f \in C([-1,0];L^2(B_1))\cap L^2(-1,0;H^1(B_1))$, with $\partial_t f \in L^2(-1,0;(H^1)'(B_1))$, $| U(t,\bxi)| \leq 1$ a.e.~$(t,\bxi) \in Q_1$, and $\Vert V \Vert_{L^q(Q_1)} \leq 1$ for $q>5$, be a weak subsolution of \eqref{eq:rescaled eqn}. Let $\eta \in C^\infty_c(B_1)$ be any compactly supported function independent of $t$, and $K  \geq 0$. Define $v = (f-K)_+$. Then, there exists a positive constant $C_{ }$, independent of $\eta,K,f,U,V$, such that there holds, for all $-1 < s < t < 0$, 
    \begin{equation*}
        \begin{aligned}
            \bigg(\int_{B_1} |\eta v|^2 \d \bxi\bigg)(t) - &\bigg(\int_{B_1} |\eta v|^2 \d \bxi\bigg)(s) +\int_s^t \int_{B_1} |\nabla_{\bxi} (\eta v)|^2 \d \bxi \d \tau  \\ 
            \leq& \, C_{ }  (1+K^2)  \int_s^t \int_{B_1} (\eta + |\nabla_{\bxi} \eta |)^2(1+|V|^2) (v^2 + \mathds{1}_{\{v>0\}}) \d \bxi \d \tau. 
        \end{aligned}
    \end{equation*}
\end{lemma}

\begin{proof}
A standard argument shows (\textit{e.g.}~\cite[Theorem 2.1.11]{ziemer2012weakly}), using the compact support of $\eta$ with respect to the space-angle variable $\bxi$, that $\eta^2 v \in C([-1,0];L^2(B_1)) \cap L^2(-1,0;H^1(B_1))$ may be approximated by elements of $C^\infty_c(Q_1)$. In turn, we may insert $\eta^2 v$ into the weak subsolution formulation \eqref{eq:rescaled eqn int form}. Using also the relations $v f = v^2 + K v$, $\nabla_{\bxi} v = \mathds{1}_{v \geq 0} \nabla_{\bxi} v$, $f \mathds{1}_{v \geq 0} = (v+K)\mathds{1}_{v \geq 0}$, and $\nabla_{\bxi}(\eta^2 v) \cdot \nabla_{\bxi} v = |\nabla_{\bxi}(\eta v)|^2 - v^2 |\nabla_{\bxi} \eta|^2$, we obtain 
\begin{equation}\label{eq:prior to integrating caccioppoli}
    \begin{aligned}
    \frac{1}{2}\frac{\der}{\der t}& \int_{B_1} \eta^2 v^2 \d \bxi + \int_{B_1} |\nabla_{\bxi} (\eta v) |^2 \d \bxi \\ 
    \leq& \int_{B_1}  v^2 |\nabla_{\bxi} \eta|^2 \d \bxi +    \int_{B_1} \eta^2  v\nabla_{\bxi} v \cdot U \d \bxi + K   \int_{B_1} \eta^2  \nabla_{\bxi} v \cdot U \d \bxi + \int_{B_1} \nabla_{\bxi}(\eta v) \cdot V \eta \d\bxi \\ 
    &+ 2   \int_{B_1} v^2 \eta \nabla_{\bxi} \eta \cdot U \d \bxi + 2K   \int_{B_1} v \eta \nabla_{\bxi} \eta \cdot U \d \bxi + \int_{B_1}  \nabla_{\bxi} \eta \cdot V v \eta \d \bxi . 
    \end{aligned}
\end{equation}
Using the bound on $U$ and the Cauchy--Young inequality, also writing $v \leq \frac{1}{2}(\mathds{1}_{\{v>0\}}+v^2)$, it follows that 
\begin{equation*}
    \begin{aligned}
    \frac{1}{2}\frac{\der}{\der t} \int_{B_1} \eta^2 v^2 \d \bxi + \frac{1}{2}\int_{B_1}  |\nabla_{\bxi} (\eta v) |^2 \d \bxi  \leq&    \int_{B_1}  \eta^2  v |\nabla_{\bxi} v| \d \bxi + K   \int_{B_1} \eta^2  |\nabla_{\bxi} v | \d \bxi \\ 
    &+ C_{ } (1+K)  \int_{B_1} (\eta + |\nabla_{\bxi} \eta |)^2 (v^2 + \mathds{1}_{\{v>0\}}) \d \bxi \\ 
    &+ C\int_{B_1} \eta |\nabla_{\bxi} \eta | |V| v \d\bxi + C\int_{B_1} \eta^2 |V|^2 \mathds{1}_{\{v>0\}} \d\bxi, 
    \end{aligned}
\end{equation*}
for some universal constant $C_{ }$, whence, using the relation $\eta \nabla_{\bxi}v = \nabla_{\bxi} (\eta v) - v \nabla_{\bxi}\eta$ and the Cauchy--Young inequality to estimate the first term on the right-hand side of the previous inequality, there holds 
\begin{equation*}
    \begin{aligned}
    \frac{\der}{\der t} \int_{B_1} \eta^2 v^2 \d \bxi + \int_{B_1}  |\nabla_{\bxi} (\eta v) |^2 \d \bxi  \leq C (1+K^2)  \int_{B_1} (\eta + |\nabla_{\bxi} \eta |)^2(1+|V|^2) (v^2 + \mathds{1}_{\{v>0\}}) \d \bxi, 
    \end{aligned}
\end{equation*}
where we also used $|V| \leq \frac{1}{2}(1+|V|^2)$. Integrating the final inequality with respect to the time variable gives the result. 
\end{proof}

\subsubsection{Interior local boundedness on subcylinders}

The goal of this subsection is to prove the following result. 
\begin{prop}[Interior Local Boundedness]\label{prop:interior local boundedness proposition}
    Let $f$ be a weak subsolution of \eqref{eq:rescaled eqn}. Let $(t, \bxi) \in (0,T)\times\mathbb{R}^3$, and $r$ satisfy the constraint: 
    \begin{equation}\label{eq:r criterion local boundedness}
        0 < r < \min\big\{1,\sqrt{t/2}\big\}. 
    \end{equation}
    Then, there exists a positive constant $C$, independent of $r,(t, \bxi)$, such that there holds 
    \begin{equation}\label{eq:boundedness est loc i}
        \Vert f \Vert_{L^\infty(Q_{r}(t, \bxi))} \leq C(1+ r^{-\frac{3}{2}} )\Vert f \Vert_\mathscr{P}. 
    \end{equation}
\end{prop}

We begin by proving the following lemma.

\begin{lemma}\label{lem:interior local boundedness weak sol}
Let $q > 5$ be fixed. There exists $\delta_*>0$, depending only on $T,\Upsilon,q$, such that: for all weak subsolutions $f$ of \eqref{eq:rescaled eqn}, where $f \in C([-1,0];L^2(B_1))\cap L^2(-1,0;H^1(B_1))$, with $\partial_t f \in L^2(-1,0;(H^1)'(B_1))$, $|U| \leq 1$ a.e.~in $Q_1$, and $\Vert V \Vert_{L^{q}(Q_1)}\leq 1$, if 
\begin{equation*}
  \esssup_{t \in [-1,0]}\int_{B_1} |f(t)|^2 \d \bxi + \int_{Q_1} |\nabla_{\bxi} f|^2 \d \bxi \d t \leq \delta_*, 
\end{equation*}
then 
\begin{equation*}
   f_+ \leq \frac{1}{2} \quad \text{in } Q_{\frac{1}{2}}. 
\end{equation*}
\end{lemma}

\begin{proof}

The proof is divided in several steps. 

    \smallskip 
    \noindent 1. \textit{Iterative set-up}: Consider the sequence of times $T_k = -\frac{1}{2}(1+2^{-k})$ as well as the sequence of cylinders $\tilde{Q}_k = (T_k,0) \times \tilde{B}_k$, where $\tilde{B}_k = \{\bxi : |\bxi| < \frac{1}{2}(1+2^{-k})\}$, and define the truncations $\mathcal{T}_k f = (f-C_k)_+$ with $C_k = \frac{1}{2}(1-2^{-k})$. We consider a family of non-negative cut-off functions $\{\eta_k\}_{k\in\mathbb{N}}$, compactly supported in $\tilde{B}_{k+1}$, identically equal to $1$ in $\tilde{B}_k$, and such that $|\nabla_{\bxi}\eta_k| \leq C 2^k$ for some positive universal constant $C$. Correspondingly, we define, for all $k\in\mathbb{N}$, 
    \begin{equation*}
       \begin{aligned}
           &\mathscr{E}_k := \esssup_{t \in [T_k,0]}\bigg( \int_{B_1} |\eta_k \mathcal{T}_k f |^2 \d \bxi \bigg)(t) + \int_{T_k}^0 \int_{B_1} |\nabla_{\bxi}(\eta_k \mathcal{T}_k f)|^2 \d\bxi \d t, \\ 
           & \mathscr{E}_0 := \esssup_{t \in [-1,0]}\int_{B_1} |f_+(t)|^2 \d \bxi + \int_{Q_1} |\nabla_{\bxi} f_+|^2 \d t. 
       \end{aligned} 
    \end{equation*}

    \smallskip

    \noindent 2. \textit{Non-linear recursive estimate}: Our goal is to prove the non-linear recursive estimate: 
    \begin{equation}\label{eq:key est for iteration}
        \begin{aligned}
            \mathscr{E}_{k+1}   \leq& C_*^k \mathscr{E}_k^{1+\epsilon}, 
        \end{aligned}
    \end{equation}
    for some positive universal constant $C_*$ depending only on $\Upsilon,T$, where $$\epsilon := 1 - \frac{2}{q} - \frac{3}{5} \in (0,1)$$ depends only on $q>5$. By substituting $\eta = \eta_{k+1}$ and $K=C_k$ into the inequality of Lemma \ref{lem:energy ineq weak sol prior to local boundedness}, and constraining $T_k \leq s \leq T_{k+1} \leq t \leq 0$, we get 
    \begin{equation*}
        \begin{aligned}
            \bigg(\int_{B_1} |&\eta_{k+1} \mathcal{T}_{k+1} f|^2 \d \bxi\bigg)(t) + \int_{T_{k+1}}^t \int_{B_1} |\nabla_{\bxi} (\eta_{k+1} \mathcal{T}_{k+1} f)|^2 \d \bxi \d \tau  \\ 
            \leq& \bigg(\int_{B_1} |\eta_{k+1} \mathcal{T}_{k+1} f|^2 \d \bxi\bigg)(s) \\ 
            &+ C  (1+C_k^2)  \int_{T_{k}}^0 \int_{B_1} (\eta_{k+1} + |\nabla_{\bxi} \eta_{k+1} |)^2 (1+|V|^2)(\mathcal{T}_{k+1} f^2 + \mathds{1}_{\{\mathcal{T}_{k+1} f>0\}}) \d \bxi \d \tau, 
        \end{aligned}
    \end{equation*}
    which, by integrating the entire inequality in $s$ over the interval $[T_k , T_{k+1}]$ and noting $T_{k+1}-T_k = 2^{-(k+2)}$, yields 
        \begin{equation}\label{eq:long with k factors}
        \begin{aligned}
            \bigg(\int_{B_1}& |\eta_{k+1} \mathcal{T}_{k+1} f|^2 \d \bxi\bigg)(t) + \int_{T_{k+1}}^t \int_{B_1} |\nabla_{\bxi} (\eta_{k+1} \mathcal{T}_{k+1} f)|^2 \d \bxi \d \tau  \\ 
            \leq& \, 2^{k+2}\int_{T_k}^{T_{k+1}} \int_{B_1} |\eta_{k+1} \mathcal{T}_{k+1} f|^2 \d \bxi \d \tau \\ 
            &+ C \int_{T_{k}}^0 \int_{B_1} (\underbrace{1 + 2^{k+1}}_{\leq 2\cdot 2^{k+1}})^2 \mathds{1}_{\tilde{B}_k} (1+|V|^2)((\mathcal{T}_{k+1} f)^2 + \mathds{1}_{\{\mathcal{T}_{k+1} f>0\}}) \d \bxi \d \tau \\ 
            \leq& \,  C^k \!\!\! \int_{T_k}^{T_{k+1}}\!\!\! \int_{B_1}\!\!\! |\eta_{k+1} \mathcal{T}_{k+1} f|^2 \d \bxi \d \tau \! + C_{ }^k \!\! \underbrace{\int_{T_{k}}^0 \!\int_{B_1}\!\!\! \mathds{1}_{\tilde{Q}_k \cap \{\mathcal{T}_{k+1} f>0\}}(1+|V|^2) ((\mathcal{T}_{k+1} f)^2 + 1) \d \bxi \d \tau}_{=:I}, 
        \end{aligned}
    \end{equation}
    where the value of $C_{ }$ has changed from line to line. We proceed to estimating the term $I$. Notice that, provided $(t,\bxi) \in \tilde{Q}_k \cap \{ \mathcal{T}_{k+1} f > 0\}$, there holds 
    \begin{equation}\label{eq:when you are on the next set}
        \begin{aligned}
          \mathcal{T}_k f(t,\bxi) = f(t,\bxi) - C_k = \mathcal{T}_{k+1} f(t,\bxi) + 2^{-(k+2)} > 2^{-(k+2)}, 
        \end{aligned}
    \end{equation}
    whence, by squaring the above inequality (noting that all quantities are non-negative), we get $$\mathds{1}_{\{\mathcal{T}_{k+1} f>0\}} \leq 2^{2(k+2)} (\mathcal{T}_k f)^2 \mathds{1}_{\{\mathcal{T}_{k+1} f>0\}}.$$ In view of $\mathcal{T}_{k+1} f \leq \mathcal{T}_k f$, we therefore estimate the final term of \eqref{eq:long with k factors} as 
    \begin{equation*}
        \begin{aligned}
            I \leq & \int_{T_{k}}^0 \int_{B_1} \mathds{1}_{\tilde{Q}_k \cap \{\mathcal{T}_{k+1} f>0\}}(1+|V|^2) (1+ 2^{2(k+2)}) (\mathcal{T}_k f)^2 \d \bxi \d \tau \\ 
            \leq & \, C^k \int_{T_{k}}^0 \int_{B_1} \mathds{1}_{\tilde{Q}_k \cap \{\mathcal{T}_{k+1} f>0\}}(1+|V|^2) (\mathcal{T}_k f)^2 \d \bxi \d \tau, 
        \end{aligned}
    \end{equation*}
    and there holds 
       \begin{equation*}
        \begin{aligned}
            \mathscr{E}_{k+1}   \leq& \, C^k \int_{T_k}^{T_{k+1}} \int_{B_1} |\eta_{k+1} \mathcal{T}_{k+1} f|^2 \d \bxi \d \tau + C_{ }^k \int_{T_{k}}^0 \int_{B_1} \mathds{1}_{\tilde{Q}_k \cap \{\mathcal{T}_{k+1} f>0\}}(1+|V|^2) |\mathcal{T}_k f|^2 \d \bxi \d \tau. 
        \end{aligned}
    \end{equation*}
    Similarly, using the boundedness of the cut-off as well as $\mathcal{T}_{k+1} f \leq \mathcal{T}_k f$, the first term on the right-hand side of the previous estimate may be rewritten as 
    \begin{equation*}
        \begin{aligned}
            \int_{T_k}^{T_{k+1}} \int_{B_1} |\eta_{k+1} \mathcal{T}_{k+1} f|^2 \d \bxi \d \tau \leq \int_{T_k}^0 \int_{B_1} \mathds{1}_{\tilde{Q}_k \cap \{ \mathcal{T}_{k+1} f>0\}} |\mathcal{T}_{k+1} f|^2 \d \bxi \d \tau \leq I, 
        \end{aligned}
    \end{equation*}
    and thus 
           \begin{equation}\label{eq:almost there E k plus 1}
        \begin{aligned}
            \mathscr{E}_{k+1}   \leq& \, C_{ }^k \int_{T_{k}}^0 \int_{B_1} \mathds{1}_{\tilde{Q}_k \cap \{\mathcal{T}_{k+1} f>0\}}(1+|V|^2) |\mathcal{T}_k f|^2 \d \bxi \d \tau. 
        \end{aligned}
    \end{equation}
    Using the Sobolev inequality, there holds, for a positive constant $C$ independent of $t,k$, 
    \begin{equation*}
        \Vert \eta_k \mathcal{T}_k f(t,\cdot) \Vert_{L^6({B_1})}^2 \leq C\Vert \eta_k \mathcal{T}_k f(t,\cdot) \Vert_{H^1({B_1})}^2, 
    \end{equation*}
    whence $\Vert \eta_k \mathcal{T}_k f \Vert_{L^2(T_k,0;L^6({B_1}))} \leq C \mathscr{E}_k^{\frac{1}{2}}$, and 
    \begin{equation}\label{eq:est pre markov}
        \begin{aligned}
            \Vert \mathcal{T}_k f \Vert^2_{L^2(\tilde{Q}_k)} = \Vert \eta_k \mathcal{T}_k f \Vert^2_{L^2(\tilde{Q}_k)} &\leq \int_{T_k}^0 \bigg( \int_{B_1} |\eta_k \mathcal{T}_k f(t,\bxi)|^2 \d \bxi \bigg) \d t \\ 
            &\leq \int_{T_k}^0 |{B_1}|^{\frac{2}{3}}\bigg( \int_{B_1}|\eta_k \mathcal{T}_k f(t,\bxi)|^6 \d \bxi \bigg)^{\frac{1}{3}} \d t \\ 
            &= C \Vert \eta_k \mathcal{T}_k f \Vert^2_{L^2(T_k,0;L^6({B_1}))} \\ 
            &\leq C\mathscr{E}_k. 
        \end{aligned}
    \end{equation}
   We remark that this estimate alone would be insufficient for bounding the first term on the right-hand side of \eqref{eq:long with k factors}; indeed, we must have a \emph{non-linear} estimate in order for the iteration procedure to succeed.

Meanwhile, $\Vert \eta_k \mathcal{T}_k f \Vert_{L^\infty(T_k,0;L^2({B_1}))} \leq \mathscr{E}_k^{\frac{1}{2}}$. We interpolate between these two norms. More precisely, using the Interpolation Lemma \ref{lem:dibenedetto classic}, there exists a positive constant $C$, independent of $k$, such that 
    \begin{equation*}
        \begin{aligned}
            \Vert \eta_k \mathcal{T}_k f \Vert_{L^p(\tilde{Q}_k)} \leq C(1+|T_k|) \Big( \Vert \eta_k \mathcal{T}_k f \Vert_{L^2(T_k,0;H^1({B_1}))} + \Vert \eta_k \mathcal{T}_k f \Vert_{L^\infty(T_k,0;L^2({B_1}))} \Big) \leq C \mathscr{E}_k^{\frac{1}{2}}, 
        \end{aligned}
    \end{equation*}
    where $p = \frac{2}{3}(2+3) = \frac{10}{3}$, \textit{i.e.}, $\Vert \eta_k \mathcal{T}_k f \Vert_{L^{\frac{10}{3}}(\tilde{Q}_k)}^2 \leq C \mathscr{E}_k$, from which we deduce 
    \begin{equation*}
        \Vert (\mathcal{T}_k f)^2 \Vert_{L^{\frac{5}{3}}(\tilde{Q}_k)} \leq \Vert \eta_k \mathcal{T}_k f \Vert_{L^{\frac{10}{3}}(\tilde{Q}_k)}^2 \leq C \mathscr{E}_k. 
    \end{equation*}

In turn, returning to \eqref{eq:almost there E k plus 1} and using the H\"older, Jensen, and Minkowski inequalities with the assumption $\Vert V \Vert_{L^q(Q_1)}\leq 1$, there holds 
    \begin{equation*}
        \begin{aligned}
            \mathscr{E}_{k+1}  &\leq C^k |\tilde{Q}_k \cap \{\mathcal{T}_{k+1} f>0\}|^{1-\frac{2}{q}-\frac{3}{5}} \Vert 1 + |V|^2 \Vert_{L^q(Q_1)} \Vert (\mathcal{T}_k f)^2 \Vert_{L^{\frac{5}{3}}(\tilde{Q}_k)}   \\ 
            &\leq C_{ }^k |\tilde{Q}_k \cap \{\mathcal{T}_{k+1} f>0\}|^{1-\frac{2}{q}-\frac{3}{5}}\mathscr{E}_k; 
        \end{aligned}
    \end{equation*}
   note that the application of H\"older's inequality is justified due to the condition $V \in L^q(Q_1)$ for $q>5$, \textit{i.e.}, $\epsilon = 1-2/q-3/5 \in (0,1)$. By applying the Markov inequality, using \eqref{eq:when you are on the next set} to write $\tilde{Q}_k \cap \{\mathcal{T}_{k+1} f>0\} \subset \{ |\eta_k \mathcal{T}_k f|^2 > 2^{-2(k+2)} \}$, we estimate, using also the bound \eqref{eq:est pre markov}, 
    \begin{equation*}
        |\tilde{Q}_k \cap \{\mathcal{T}_{k+1} f> 0\}| \leq 2^{2(k+2)} \Vert \eta_k \mathcal{T}_k f \Vert^2_{L^2(\tilde{Q}_k)} \leq C^k \mathscr{E}_k, 
    \end{equation*}
    whence we get the desired non-linear estimate \eqref{eq:key est for iteration}.

\smallskip

    \noindent 3. \textit{Initialisation and iterative procedure}: Applying the standard iteration lemma \cite[\S 1, Lemma 4.1]{DiBenedetto} to the recursive estimate relation \eqref{eq:key est for iteration}, we deduce that there exists $\delta_* = \delta(C_*,\epsilon)>0$ sufficiently small such that if $\mathscr{E}_0 \leq \delta_*$, then $\lim_{k\to\infty}\mathscr{E}_k = 0$. The Monotone Convergence Theorem and \eqref{eq:est pre markov} then imply 
    \begin{equation*}
        \int_{Q_{\frac{1}{2}}} \big(f - \frac{1}{2}\big)_+^2 \d \bxi \d t = \lim_{k\to\infty}\Vert \mathcal{T}_k f \Vert^2_{L^2(\tilde{Q}_k)} \leq \lim_{k\to\infty}\mathscr{E}_k = 0, 
    \end{equation*}
    which yields the conclusion of the lemma. 
\end{proof}

Proposition \ref{prop:interior local boundedness proposition} now follows as a simple corollary of Lemma \ref{lem:interior local boundedness weak sol} by a standard scaling argument using Lemma \ref{lem:rescaling i}.

\begin{proof}[Proof of Proposition \ref{prop:interior local boundedness proposition}]
    
Let $f_r,U_r$ be defined from $f$ as per equation \eqref{eq:rescaling i} in the proof of Lemma \ref{lem:rescaling i}, with $\delta$ chosen to be the specific value $\delta_*$. We then apply Lemma \ref{lem:interior local boundedness weak sol} to $f_r$, from which we obtain 
\begin{equation}\label{eq:bound on frplus}
    (f_r)_+ \leq \frac{1}{2} \quad \text{in } Q_{\frac{1}{2}}.
\end{equation}
Then, using the non-negativity of $f$ and the rescaling \eqref{eq:rescaling i} to transfer the bound \eqref{eq:bound on frplus}, we obtain the result.   
\end{proof}

\subsubsection{Boundedness away from initial time}

The smallness of the radius $r$ of the subcylinder $Q_r(t,\bxi)$ constrains the result of Proposition \ref{prop:interior local boundedness proposition} to being local in the interior. However, the periodicity of the problem actually means that the result is global-in-space, while being local-in-time; this is manifestly clear from the fact that the constraint on the size of $r$ in $Q_r(t,\bxi)$ depends only on the coordinate $t$, and not on $\bxi$, as is shown in the criterion \eqref{eq:r criterion local boundedness}. 

Our objective is therefore to extend the result of Proposition \ref{prop:interior local boundedness proposition} from subcylinders to infinite strips away from the initial time, \textit{i.e.}~Theorem \ref{thm:global boundedness away from initial time}, which is proved by an exhaustion argument.

\begin{proof}[Proof of Proposition \ref{thm:global boundedness away from initial time}]
    Fix $t \in (0,T)$. Define $r_t:= \frac{1}{2}\min\{1,\sqrt{t/2}\}$. Given this radius $r_t$, select $\{\bxi_1,\dots,\bxi_N\}$ to be any finite collection of points in $\Upsilon$ chosen such that 
    \begin{equation*}
        \Upsilon \subset \bigcup_{j=1}^{N} B_{r_t}(\bxi_j); 
    \end{equation*}
    it is clear that such a collection exists, and an easy argument shows that we may take $N = \ceil{C r_t^{-3}}$ for some positive constant $C$ depending only on $\Upsilon$. Similarly, let $\{t_0,\dots,t_M\}$ be any collection of times such that $t = t_0 < t_1 < \dots < t_M = T$ such that $|t_{j}-t_{j-1}|<r_t^2$ for all $j \in \{1,\dots,M\}$; as before, one may take $M=\ceil{C r_t^{-2}}$ for a suitable constant $C$ depending only on $T$. It follows that 
    \begin{equation*}
        (t,T) \times \Upsilon \subset \bigcup_{i=0}^M\bigcup_{j=1}^N Q_{r_t}(t_i,\bxi_j), 
    \end{equation*}
    and thus 
    \begin{equation*}
        \begin{aligned}
            \Vert f \Vert_{L^\infty((t,T)\times\Upsilon)} \leq \sum_{i=0}^M \sum_{j=1}^N \Vert f \Vert_{L^\infty(Q_{r_t}(t_i,\bxi_j))}. 
        \end{aligned}
    \end{equation*}
    Observe that, for each subcylinder $Q_{r_t}(t_i,\bxi_j)$ there holds $r_t < \min\{1,\sqrt{t_i/2}\}$, whence we may apply Proposition \ref{prop:interior local boundedness proposition} to get 
    \begin{equation*}
        \Vert f \Vert_{L^\infty((t,T)\times\Upsilon)} \leq \underbrace{MN}_{\leq C r_t^{-5}} (1+r_t^{-\frac{3}{2}}) \Vert f \Vert_\mathscr{P}. 
    \end{equation*}
    The result then follows from the definition of $r_t$. 
\end{proof}

\subsection{Global-in-time boundedness for $L^\infty$ initial data}\label{subsec:alikakos}

In this section we prove Theorem \ref{thm:alikakos}, concerning the global-in-time boundedness of weak solutions, assuming initial data in $L^\infty$. Our strategy is based on \cite[\S 3]{KimZhangII}. We note that this section is separate from the rest of the regularity analysis. 

Before proceeding to the proof of Theorem \ref{thm:alikakos}, we recall a version of a technical lemma from \cite{KimZhangII}; we omit the proof, which can be found in \cite[Appendix A]{KimZhangII}. 

\begin{lemma}[Lemma 3.2 of \cite{KimZhangII}]\label{lem:iteration kz}
    Let $A_k:[0,\infty) \to [0,\infty)$ be a sequence of functions satisfying the differential inequality 
    \begin{equation*}
        \frac{\der}{\der t} A_k + C_0 A_k \leq C_1^k (A_{k-1})^2 \qquad \text{for all } t, k, 
    \end{equation*}
 for some positive constants $C_0,C_1$, and assume that $A_0(t)$ is uniformly bounded in time and $\{A_k(0)\}_k$ is uniformly bounded in $k$. Then, with $n_k = 2^k$, the sequence $\{A_k^{1/n_k}(t)\}_k$ is uniformly bounded in time. 
\end{lemma}

\begin{proof}[Proof of Theorem \ref{thm:alikakos}]
    Let $n \geq 2$. We test the equation \eqref{eq:main eqn} against $nf^{n-1}$ and, using the boundedness of $\rho$, obtain 
    \begin{equation*}
        \frac{\der}{\der t}\int_\Upsilon f^{n} \d \bxi + \frac{4(n-1)}{n} \min\{D_e,1\} \int_\Upsilon |\nabla_{\bxi} f^{\frac{n}{2}}|^2 \d \bxi \leq 2(n-1)\Pe \int_\Upsilon f^{\frac{n}{2}} |\nabla f^{\frac{n}{2}}| \d \bxi, 
    \end{equation*}
    which, by applying Young's inequality and using the lower bound $4(n-1)/n \geq 2$, implies 
   \begin{equation*}
        \frac{\der}{\der t}\int_\Upsilon f^{n} \d \bxi + \frac{3}{2} \min\{D_e,1\} \int_\Upsilon |\nabla_{\bxi} f^{\frac{n}{2}}|^2 \d \bxi \leq \frac{2(n-1)^2\Pe^2}{\min \{D_e,1\}} \Vert f^{\frac{n}{2}} \Vert^2_{L^2(\Upsilon)}. 
    \end{equation*} 
    By applying the Gagliardo--Nirenberg inequality, we get 
    \begin{equation*}
      \begin{aligned}  \Vert f^{\frac{n}{2}} \Vert_{L^2(\Upsilon)} &\leq C_{GN} \Big( \Vert \nabla_{\bxi} f^{\frac{n}{2}} \Vert_{L^2(\Upsilon)}^{\frac{3}{5}} \Vert f^{\frac{n}{2}} \Vert^{\frac{2}{5}}_{L^1(\Upsilon)} + \Vert f^{\frac{n}{2}} \Vert_{L^1(\Upsilon)} \Big), 
      \end{aligned}
    \end{equation*}
   from which, by using Young's inequality, we deduce 
   \begin{equation}\label{eq:lower bound gn}
    \Vert \nabla_{\bxi} f^{\frac{n}{2}} \Vert_{L^2(\Upsilon)} 
 \geq  \frac{1}{C_{GN}'} \Vert f^{\frac{n}{2}} \Vert_{L^2(\Upsilon)} - \Vert f^{\frac{n}{2}} \Vert_{L^1(\Upsilon)}, 
   \end{equation}
   with $C_{GN}'=\frac{7}{5}C_{GN}$, as well as 
   \begin{equation*}
       \begin{aligned}  \Vert f^{\frac{n}{2}} \Vert_{L^2(\Upsilon)}^2 &\leq \epsilon^{\frac{5}{3}}\frac{6 C_{GN}^2}{5} \Vert \nabla_{\bxi} f^{\frac{n}{2}} \Vert_{L^2(\Upsilon)} + 2C_{GN}^2\Big( 1 + \frac{2}{5 \epsilon^{\frac{5}{2}}} \Big) \Vert f^{\frac{n}{2}} \Vert_{L^1(\Upsilon)}^2, 
      \end{aligned}
   \end{equation*}
for all $\epsilon>0$. Consequently, by setting 
$$\epsilon = \left(\frac{5 \min\{D_e,1\}^2}{24(n-1)^2 \Pe^2 C_{GN}^2} \right)^{3/5},$$
we get 
  \begin{equation*}
  \begin{aligned}   
  \frac{\der}{\der t}\int_\Upsilon f^{n} \d \bxi + \min\{D_e,1\} \Vert \nabla_{\bxi} f^{\frac{n}{2}} \Vert_{L^2(\Upsilon)}^2 &\leq (n-1)^2\frac{4 \Pe^2 C_{GN}^2}{\min\{D_e,1\}} \Big( 1 + \frac{2}{5 \epsilon^{\frac{5}{2}}} \Big) \Vert f^{\frac{n}{2}} \Vert_{L^1(\Upsilon)}^2. 
        \end{aligned}
    \end{equation*}
Thus, by defining 
    \begin{equation*}
        c_0 := \frac{\min\{D_e,1\}}{2(C_{GN}')^2}, \quad c_1 := \min\{D_e,1\} + \frac{4(n-1)^2 \Pe^2 C_{GN}^2}{\min\{D_e,1\}} \Big( 1 + \frac{2}{5 \epsilon^{\frac{5}{2}}} \Big) , 
    \end{equation*}
   using the lower bound \eqref{eq:lower bound gn}, we obtain 
      \begin{equation*}   
  \frac{\der}{\der t} \Vert f^{n} \Vert_{L^1(\Upsilon)} + c_0 \Vert f^n \Vert_{L^1(\Upsilon)} \leq c_1 n^2 \Vert f^{\frac{n}{2}} \Vert^2_{L^1(\Upsilon)}, 
    \end{equation*}
    whence, by setting $n=2^k$ and $$A_k(t) := \Vert f^{n_k}(t,\cdot) \Vert_{L^1(\Upsilon)}, $$ 
    we get 
    \begin{equation*}
        \frac{\der}{\der t} A_k + c_0 A_k \leq c_1 4^k (A_{k-1})^2. 
    \end{equation*}
Using the boundedness of the initial data $f_0 \in L^\infty(\Upsilon)$ and applying Lemma \ref{lem:iteration kz}, we deduce global-in-time estimate 
    \begin{equation*}
        \Vert f(t,\cdot) \Vert_{L^{2^k}(\Upsilon)} \leq C(c_0,c_1,\Vert f_0 \Vert_{L^\infty}) \qquad \text{for all } t, k, 
    \end{equation*}
    whence the result follows. 
\end{proof}

\section{Higher Regularity of Weak Solutions}\label{sec:higher reg}

We obtain smoothness of all weak solutions $f$ away from the initial time. Our strategy is to take derivatives in time in the equation \eqref{eq:main eqn}, and then to apply the iteration procedure used to prove Theorem \ref{thm:global boundedness away from initial time} to the resulting equation. For clarity of exposition, in \S \ref{sec:C1 away from initial time} we first show how to do this for the first derivative in time, denoted by $\dot{f}:=\partial_t f$ with $\dot{\rho}$ defined analogously, before moving on to general higher derivatives in \S \ref{sec:smooth away from initial time}, which are denoted by $f^{(n)} := \partial_t^n f$ and $\rho^{(n)} := \partial_t^n \rho$.

The main results of this subsection are as follows, and are used to prove Theorem \ref{thm:smooth away from initial}. 

\begin{prop}[Boundedness of Time Derivatives away from Initial Time]\label{prop:boundedness of all time derivs}
    For all integer $n$, there exists a decreasing positive function $\psi^n :(0,T) \to (0,\infty)$ such that $$\limsup_{t \to 0^+}\psi^n(t) = +\infty$$ and, for a.e.~$t \in (0,T)$, there holds 
    \begin{equation}
        \Vert f^{(n)} \Vert_{L^\infty((t,T)\times\Upsilon)} \leq \psi^n(t). 
    \end{equation}
\end{prop}

We will then use the previous proposition and the computations developed in its proof to prove the next result: 

\begin{prop}[Sobolev Estimates for Time Derivatives away from Initial Time]\label{prop:sobolev of all time derivs}
    For all integer $n$, there exists a decreasing positive function $\Psi^n :(0,T) \to (0,\infty)$ such that $$\limsup_{t \to 0^+}\Psi^n(t) = +\infty$$ and, for a.e.~$t \in (0,T)$, there holds 
    \begin{equation}
        \Vert f^{(n)} \Vert_{W^{2,\frac{10}{3}}((t,T)\times\Upsilon)} \leq \Psi^n(t). 
    \end{equation}
\end{prop}

Theorem \ref{thm:smooth away from initial} then follows as an immediate corollary of Proposition \ref{prop:sobolev of all time derivs}, as shown below. 

\begin{proof}[Proof of Theorem \ref{thm:smooth away from initial}]
Morrey's inequality and Proposition \ref{prop:sobolev of all time derivs} implies that $f^{(n)} \in C^{\frac{4}{5}}((t,T)\times\Upsilon)$, where we used that the dimension of $(t,T)\times\Upsilon$ is $4< 2\cdot 10/3$. It then follows that, for all integers $n$, the function $f^{(n)}$ is continuous on the subset $(t,T)\times \Upsilon$. By returning to the equation and differentiating, a straightforward argument shows that the continuity of $\{f^{(n)}\}_n$ implies continuity of the derivatives with respect to $\bxi$ of all orders, and analogously for all mixed-derivatives in $t$ and $\bxi$. The proof is complete. 
\end{proof}

\subsection{Boundedness of $\dot{f}$ away from initial time}\label{sec:C1 away from initial time}

The goal of this subsection is to prove the case $n=1$ in Proposition \ref{prop:boundedness of all time derivs}. Our underlying strategy is to derive an equation for $\dot{f}$ to which the De Giorgi method can be applied. First, we must derive a $H^2$-type bound on the solution away from the initial time, which is the content of the next lemma.

\begin{lemma}[$H^2$-type Estimate for $f$ away from Initial Time]\label{lem:H2 bound for C1}
   There exists a positive constant $C$, independent of $t$, such that for a.e.~$t \in (0,T)$ there holds 
    \begin{equation*}
       \begin{aligned} \Vert \nabla_{\bxi}&f\Vert^2_{L^\infty(t,T;L^2(\Upsilon))} + \Vert \Delta_{\bxi} f\Vert^2_{L^2((t,T)\times\Upsilon)} \\ 
       &\leq C \big(1+\Vert \nabla_{\bxi}f(t,\cdot) \Vert_{L^2(\Upsilon)}^2\big) \big(1+\Vert f \Vert_{L^\infty((t,T)\times\Upsilon)}^2\big)^2 \exp\Big( T\big( 1+\Vert f \Vert_{L^\infty((t,T)\times\Upsilon)}^2 \big) \Big), 
       \end{aligned}
    \end{equation*}
    and 
    \begin{equation*}
        \Vert \dot{f} \Vert_{L^2((t,T)\times\Upsilon)} \leq C \Big( \Vert \nabla_{\bxi}f(t,\cdot) \Vert_{L^2(\Upsilon)} + \big(1+\Vert f \Vert_{L^\infty((t,T)\times\Upsilon)}\big) \Vert \nabla f \Vert_{L^2(0,T;H^1(\Upsilon))} \Big). 
    \end{equation*}
\end{lemma}
Remark that $\nabla_{\bxi} f \in L^2(\Upsilon_T)$ implies, using Markov's inequality, that $\Vert \nabla_{\bxi} f(t,\cdot)\Vert_{L^2(\Upsilon)}^2$ is finite a.e.~in $(0,T)$; using also Proposition \ref{thm:global boundedness away from initial time}, the right-hand sides of the previous estimates are therefore finite. Furthermore the boundedness of $\Vert \Delta_{\bxi} f \Vert_{L^2((t,T)\times\Upsilon)}$ yields an identical estimate for the full Hessian $\Vert \nabla^2_{\bxi} f \Vert_{L^2((t,T)\times\Upsilon)}$ by virtue of Lemma \ref{lem:CZ periodic}.

\begin{proof}
    Let $\eta \in C^\infty_c(B_1)$ be the usual non-negative bump function with unit integral $\int_{\mathbb{R}^3}\eta(\bxi) \d \bxi = 1$, and define the sequence of Friedrichs mollifiers $\eta_\varepsilon(\bxi) := \varepsilon^{-3}\eta(\bxi/\varepsilon)$. Correspondingly, define $f_\varepsilon := f(t,\cdot)*\eta_\varepsilon = \int_{\mathbb{R}^3} \eta_\varepsilon(\bxi-\bzeta) f(t,\bzeta) \d \bzeta$ and $\rho_\varepsilon := \int_0^{2\pi} f_\varepsilon \d \theta$. Note that this operation preserves the periodicity and that the convolution is well-defined as $f,\rho$ extend periodically to the full space. It is straightforward to verify that there holds 
\begin{equation}\label{eq:mollified eqn in H2 bound for C1}
    \partial_t f_\varepsilon + \dv_{\bxi}(U f_\varepsilon + E_\varepsilon)  = \Delta_{\bxi}f_\varepsilon, 
\end{equation}
    in the weak sense, where 
    \begin{equation*}
     \begin{aligned}   U = (1-\rho)\left(\begin{matrix}\e(\theta)\\0\end{matrix}\right), \qquad E_\varepsilon(t,\bxi) :=& \, (Uf)*\eta_\varepsilon - U f_\varepsilon \\ 
     =& \, \int_{\mathbb{R}^3} \eta_\varepsilon(\bxi-\bzeta) f(t,\bzeta) \big( U(t,\bzeta) - U(t,\bxi) \big) \d \bzeta, 
     \end{aligned}
    \end{equation*}
    where the convolution is taken only with respect to the space-angle variable $\bxi$. Using also the positivity of $f,\eta_\varepsilon$, there holds 
    \begin{equation}\label{eq:error bound Eeps}
        \begin{aligned}
     |U| \leq 1, \qquad       |E_\varepsilon(t,\bxi)| \leq 2f_\varepsilon(t,\bxi) \quad \text{a.e.~in } \Upsilon_T. 
        \end{aligned}
    \end{equation}
    Observe that, for all fixed $\varepsilon$, we have $\Delta_{\bxi}f_\varepsilon \in L^2((t,T)\times\Upsilon)$ for a.e.~$t \in (0,T)$, whence \eqref{eq:mollified eqn in H2 bound for C1} holds in the strong sense and we may test with this quantity to obtain 
    \begin{equation}\label{eq:fix botch i}
        \begin{aligned}
         \frac{1}{2}\frac{\der}{\der t}\int_\Upsilon |\nabla_{\bxi} f_\varepsilon|^2 \d \bxi +   \int_\Upsilon |\Delta_{\bxi} f_\varepsilon|^2 \d \bxi =& \int_\Upsilon \dv_{\bxi}(U f_\varepsilon + E_\varepsilon) \Delta_{\bxi} f_\varepsilon \d \bxi \\ 
         \leq& \int_\Upsilon |\nabla_{\bxi} f_\varepsilon| |\Delta_{\bxi} f_\varepsilon| \d \bxi + \int_\Upsilon |\dv_{\bxi}U| |f_\varepsilon| |\Delta_{\bxi} f_\varepsilon| \d \bxi \\ 
         &+ \int_\Upsilon |\dv_{\bxi} E_\varepsilon| |\Delta_{\bxi} f_\varepsilon| \d \bxi; 
        \end{aligned}
    \end{equation}
    we bound each of the three terms on the right-hand side of the above. The first is dealt with using the Young inequality, while for the second term we have 
    \begin{equation*}
        \begin{aligned}
            \int_\Upsilon |\dv_{\bxi}U| |f_\varepsilon| |\Delta_{\bxi} f_\varepsilon| \d \bxi &\leq \Vert f_\varepsilon \Vert_{L^\infty((t,T)\times\Upsilon)} \int_\Upsilon |\nabla \rho| |\Delta_{\bxi} f_\varepsilon| \d \bxi \\ 
            &\leq \frac{1}{4}\Vert \Delta_{\bxi} f_\varepsilon \Vert^2_{L^2(\Upsilon)} + \Vert f_\varepsilon \Vert_{L^\infty((t,T)\times\Upsilon)}^2 \Vert \nabla \rho \Vert_{L^2(\Upsilon)}^2. 
        \end{aligned}
    \end{equation*}
    For the final term, we write 
    \begin{equation*}
        \begin{aligned}
             \dv_{\bxi} E_\varepsilon = (\dv_{\bxi} (Uf))*\eta_\varepsilon - (\dv_{\bxi}U) f_\varepsilon - U \cdot \nabla_{\bxi} f_\varepsilon, 
        \end{aligned}
    \end{equation*}
    from which we estimate, using standard properties of mollifiers, 
    \begin{equation*}
        \Vert \dv_{\bxi} E_\varepsilon \Vert_{L^2(\Upsilon)} \leq C\Big( \Vert f \Vert_{L^\infty((t,T)\times\Upsilon)} \Vert \nabla \rho \Vert_{L^2(\Upsilon)} + \Vert \nabla_{\bxi} f \Vert_{L^2(\Upsilon)}    \Big), 
    \end{equation*}
    and hence the final term on the right-hand side of \eqref{eq:fix botch i} is dealt with again using the Cauchy--Young inequality.

    Thus, returning to \eqref{eq:fix botch i}, and using the Cauchy--Young inequality and elementary properties of the Friedrichs mollifier, we find that there exists a constant $C$, independent of $\varepsilon$, such that 
    \begin{equation*}
        \begin{aligned}
         \frac{\der}{\der t}\int_\Upsilon |\nabla_{\bxi} f_\varepsilon|^2 \d \bxi +   \int_\Upsilon |\Delta_{\bxi} f_\varepsilon|^2 \d \bxi \leq C\Big( \Vert f \Vert_{L^\infty((t,T)\times\Upsilon)}^2  + \big( 1+\Vert f \Vert_{L^\infty((t,T)\times\Upsilon)}^2 \big) \Vert \nabla_{\bxi} f \Vert_{L^2(\Upsilon)}^2    \Big). 
        \end{aligned}
    \end{equation*}
    We remind the reader that, since $\nabla_{\bxi} f \in L^2(\Upsilon_T)$, it follows from Markov's inequality that $\Vert \nabla_{\bxi} f(t,\cdot)\Vert_{L^2(\Upsilon)}^2$ is finite a.e.~in $(0,T)$. In turn, for a.e.~$t \in (0,T)$, we deduce from Gr\"onwall's Lemma that there holds 
\begin{equation*}
        \Vert \nabla_{\bxi}f_\varepsilon \Vert^2_{L^\infty(t,T;L^2(\Upsilon))} \leq C \Vert \nabla_{\bxi}f_\varepsilon(t,\cdot) \Vert_{L^2(\Upsilon)}^2 \Vert f \Vert_{L^\infty((t,T)\times\Upsilon)}^2 \exp\Big( T\big( 1+\Vert f \Vert_{L^\infty((t,T)\times\Upsilon)}^2 \big) \Big), 
    \end{equation*}
    whence 
    \begin{equation*}
       \begin{aligned} \Vert &\nabla_{\bxi}f_\varepsilon \Vert^2_{L^\infty(t,T;L^2(\Upsilon))} + \Vert \Delta_{\bxi} f_\varepsilon \Vert^2_{L^2((t,T)\times\Upsilon)} \\ 
       &\leq C \big(1+\Vert \nabla_{\bxi}f_\varepsilon(t,\cdot) \Vert_{L^2(\Upsilon)}^2\big) \big(1+\Vert f \Vert_{L^\infty((t,T)\times\Upsilon)}^2\big)^2 \exp\Big( T\big( 1+\Vert f \Vert_{L^\infty((t,T)\times\Upsilon)}^2 \big) \Big), 
       \end{aligned}
    \end{equation*}
    and we let $\varepsilon \to 0$ to deduce the estimate for $f$. Then, using the equation, there holds 
    \begin{equation}\label{eq:using the equation}
        \Vert \partial_t f \Vert_{L^2((t,T)\times\Upsilon)} \leq \Vert \Delta_{\bxi} f_\varepsilon \Vert_{L^2((t,T)\times\Upsilon)} + \Vert f \Vert_{L^\infty((t,T)\times\Upsilon)}\Vert \nabla \rho \Vert_{L^2(\Omega_T)} + \Vert \nabla f \Vert_{L^2(\Upsilon_T)}. 
    \end{equation}
We conclude the proof by bounding $\Vert\nabla \rho\Vert_{L^2(\Omega_T)}$ by $\Vert\nabla f\Vert_{L^2(\Upsilon_T)}$ using Jensen's inequality. 
\end{proof}

The following corollary of Lemma \ref{lem:H2 bound for C1} is immediate. 

\begin{cor}\label{cor: gradients in L4 rho and p}
    For a.e.~$t \in (0,T)$, there holds 
    \begin{equation*}
        \Vert \nabla \rho \Vert_{L^4((t,T)\times\Omega)} + \Vert \nabla \p \Vert_{L^4((t,T)\times\Omega)} \leq C \Big( \Vert \nabla_{\bxi}f \Vert_{L^\infty(t,T;L^2(\Upsilon))} + \Vert \Delta_{\bxi} f \Vert_{L^2(\Upsilon_T)} \Big). 
    \end{equation*}
\end{cor}
\begin{proof}
    Using Jensen's inequality and the relations between $f$ and $\rho,\p$, we obtain for a.e.~$t \in (0,T)$ 
    \begin{equation*}
      \begin{aligned}
       \big( \Vert \nabla \rho \Vert_{L^\infty(t,T;L^2(\Omega))} + \Vert \nabla^2 \rho \Vert_{L^2((t,T)\times\Omega)} \big) +& \big( \Vert \nabla \p \Vert_{L^\infty(t,T;L^2(\Omega))} + \Vert \nabla^2 \p \Vert_{L^2((t,T)\times\Omega)} \big) \\   &\leq C \Big( \Vert \nabla_{\bxi}f \Vert_{L^\infty(t,T;L^2(\Upsilon))} + \Vert \Delta_{\bxi} f \Vert_{L^2(\Upsilon_T)} \Big), 
      \end{aligned}
    \end{equation*}
    where we again used Lemma \ref{lem:CZ periodic}. The result then follows from the conclusion of Lemma \ref{lem:H2 bound for C1} and the Interpolation Lemma \ref{lem:dibenedetto classic}, where we used that $\Omega \subset \mathbb{R}^2$ has lower dimension than $\Upsilon \subset \mathbb{R}^3$. 
\end{proof}

In the next lemma we derive the equation for $\dot{f}$. 

\begin{lemma}[Equation for $\dot{f}$]\label{lem:eqn for f dot}
There exists a positive constant $C$ depending only on $T$ such that there holds 
\begin{equation}\label{eq:dt f in L2H1}
   \begin{aligned}
       \Vert \dot{f} &\Vert_{L^\infty(t,T;L^2(\Upsilon))}^2 +\Vert \dot{f} \Vert_{L^2(t,T;H^1(\Upsilon))}^2 \\ 
       & \, \, \, \,\,\,\,\,\,\,\,\,\, \leq C\big(1+ \Vert f \Vert^2_{L^\infty((t,T)\times\Upsilon)}\big)^2\exp\big({T(1+ \Vert f \Vert^2_{L^\infty((t,T)\times\Upsilon))}}\big)\big(1+\Vert \dot{f}(t,\cdot) \Vert_{L^2(\Upsilon)}^2\big). 
   \end{aligned} 
\end{equation}
Furthermore, for a.e.~$t \in (0,T)$, there holds in the weak sense 
\begin{equation}\label{eq:eqn differentiated once in time}
    \de_t \fd +  \dv \big(  (\fd (1- \rho) 
    - f \rd) \,\e(\theta)\big)
= \Delta_{\bxi} \fd. 
\end{equation}    
\end{lemma}
We emphasise that the right-hand side of the estimate \eqref{eq:dt f in L2H1} is finite by virtue of $\dot{f} \in L^2((t,T)\times\Upsilon)$ for a.e.~$t\in(0,T)$, by Lemma \ref{lem:H2 bound for C1}, whence $\Vert \dot{f}(t,\cdot) \Vert_{L^2(\Upsilon)}$ is finite for a.e.~$t \in (0,T)$.

\begin{proof}
Fix $\delta>0$ arbitrarily and extend $f$ by zero outside of $[0,T]$ to the larger time-interval $(-\delta,T+\delta)$; we note that this extends the weak formulation to the larger time-interval and preserves 
\begin{equation*}
    \Vert f \Vert_{L^2(-\delta,T+\delta;H^1(\Upsilon))} = \Vert f \Vert_{L^2(0,T;H^1(\Upsilon))}. 
\end{equation*}
We define, for $0<|h|<\delta$ and a.e.~$(t,\bxi) \in (0,T)\times\Upsilon$, the difference quotients in time: 
\begin{equation*}
    D_h f(t,\bxi) := \frac{f(t+h,\bxi) - f(t,\bxi)}{h}, 
\end{equation*}
and, correspondingly, 
\begin{equation*}
    D_h \rho(t,x) := \frac{\rho(t+h,\bxi) - \rho(t,\bxi)}{h} = \int_0^{2\pi} D_h f(t,x,\theta) \d \theta, 
\end{equation*}
whence $|D_h \rho(t,\x)| \leq \int_0^{2\pi}|D_h f(t,\x,\theta)| \d \theta$ and Jensen's inequality implies 
\begin{equation}\label{eq:diff quotient rho est in terms of f}
    \Vert D_h \rho(t,\cdot) \Vert_{L^2(\Omega)} \leq 2\pi \Vert D_h f(t,\cdot) \Vert_{L^2(\Upsilon)}. 
\end{equation} 

In what follows, we use the notation $\tau_h \rho = \rho(t+h)$; the estimate on $\rho$ implies $0 \leq \tau_h \rho \leq 1$ a.e. Direct computation show that the equation for $D_h f$ reads, for a.e.~$t \in (0,T)$, 
\begin{equation}\label{eq:time diff quot eqn i}
  \partial_t D_h f +   \dv\big( (D_h f (1-\tau_h \rho) - f D_h \rho)\e(\theta) \big) = \Delta_{\bxi} D_h f, 
\end{equation}
in the weak sense. Note that there holds $D_h f \in L^2(0,T;H^1(\Upsilon))$, whence it is an admissible test function to insert into the weak formulation. It follows that, for a.e.~$0<s<t <T$ away from the initial time, 
\begin{equation*}
    \begin{aligned}
        \frac{1}{2}\frac{\der}{\der t}\int_\Upsilon |D_h f|^2 \d \bxi +\int_\Upsilon & |\nabla_{\bxi} D_h f|^2 \d \bxi    \\ 
        &=   \int_\Upsilon (D_h f (1-\tau_h \rho) - f D_h \rho) \e(\theta) \cdot \nabla D_h f \d \bxi \\ 
        &\leq \frac{1}{2}\int_\Upsilon |\nabla_{\bxi} D_h f|^2 \d \bxi + \frac{1}{2}\big( 1 + \Vert f \Vert_{L^\infty((s,T)\times\Upsilon)}^2 \big) \int_\Upsilon |D_h f|^2 \d \bxi, 
    \end{aligned}
\end{equation*}
where we used the Young inequality and the estimate \eqref{eq:diff quotient rho est in terms of f}; we control $D_h \rho$ with $D_h f$ using Jensen's inequality and the Fubini--Tonelli Theorem. It therefore follows that 
\begin{equation}\label{eq:pre gronwall time deriv ii}
    \begin{aligned}
        \frac{\der}{\der t}\int_\Upsilon |D_h f|^2 \d \bxi + \int_\Upsilon |\nabla_{\bxi} D_h f|^2 \d \bxi 
        \leq \big( 1 + \Vert f \Vert_{L^\infty((s,T)\times\Upsilon)}^2 \big)\int_\Upsilon |D_h f|^2 \d \bxi, 
    \end{aligned}
\end{equation}
whence the Gr\"onwall Lemma implies 
\begin{equation*}
    \esssup_{[s,T]} \Vert D_h f \Vert^2_{L^2(\Upsilon)} \leq \big(1+ \Vert f \Vert^2_{L^\infty((s,T)\times\Upsilon)}\big)\exp\big({T(1+ \Vert f \Vert^2_{L^\infty((s,T)\times\Upsilon)})}\big)\Vert D_h f(s) \Vert^2_{L^2(\Upsilon)}. 
\end{equation*}
By letting $h \to 0$, it follows that $\dot{f} := \partial_t f \in L^\infty(s,T;L^2(\Upsilon))$ for a.e.~$s \in (0,T)$ and moreover 
\begin{equation}\label{eq:time deriv in Linfty L2}
     \Vert \dot{f} \Vert_{L^\infty(s,T;L^2(\Upsilon))}^2 \leq C\big(1+ \Vert f \Vert^2_{L^\infty((s,T)\times\Upsilon)}\big)\exp\big({T(1+ \Vert f \Vert^2_{L^\infty((s,T)\times\Upsilon)})}\big)\Vert \dot{f}(s) \Vert_{L^2(\Upsilon)}^2, 
\end{equation}
for some $C$ depending only on $T$. Similarly, we obtain the boundedness of $\Vert \nabla_{\bxi} \dot{f} \Vert_{L^2((t,T)\times\Upsilon)}$ away from the initial time using \eqref{eq:pre gronwall time deriv ii}, from which, using also \eqref{eq:time deriv in Linfty L2}, we deduce, for a.e.~$t \in (0,T)$, 
\begin{equation*}
    \Vert \dot{f} \Vert_{L^2(t,T;H^1(\Upsilon))}^2 \leq C(1+ \Vert f \Vert^2_{L^\infty((t,T)\times\Upsilon)})^2 \exp\big({T(1+ \Vert f \Vert^2_{L^\infty((t,T)\times\Upsilon))}}\big)\big(1+\Vert \dot{f}(t,\cdot) \Vert_{L^2(\Upsilon)}^2\big). 
\end{equation*}

Returning to the weak formulation of \eqref{eq:time diff quot eqn i}, the improved regularity $\dot{f} \in L^2(t,T;H^1(\Upsilon))$ for a.e.~$t \in (0,T)$ implies that we may rigorously take the limit as $h \to 0$ therein. We therefore rigorously differentiate the weak formulation with respect to the time variable and obtain \eqref{eq:eqn differentiated once in time} in the weak sense over the interval $(t,T)$ for a.e.~$t \in (0,T)$. 
\end{proof}

Next, we need to upgrade the integrability of $\dot{\rho}$; this is essential, as the term $f \dot{\rho} \e(\theta)$ forms the second drift term in \eqref{eq:eqn differentiated once in time} which, as per the statement of Lemma \ref{lem:interior local boundedness weak sol}, is required to belong to $L^q$ for $q>5$. To this end, we record the following lemma. 

\begin{lemma}[Improved Integrability for $\dot{\rho}$]\label{lem:L8 for dot rho}
There exists a decreasing positive function $\phi^1 : (0,T) \to (0,\infty)$ such that $\limsup_{t \to 0^+} \phi^1(t) = +\infty$ and, for a.e.~$t \in (0,T)$, there holds 
    \begin{equation*}
        \Vert \dot{\rho} \Vert_{L^8((t,T)\times\Omega)} \leq \phi^1(t). 
    \end{equation*}
\end{lemma}

\begin{proof}

The proof is divided into several steps. 

\smallskip 
    \noindent 1. \textit{Equation for $\dot{\rho}$ and interpolated integrability}: By integrating the equation \eqref{eq:eqn differentiated once in time} with respect to the angle variable, we obtain that there holds in the weak sense 
    \begin{equation}\label{eq:eqn for dot rho}
        \partial_t \dot{\rho} + \dv \big( (1-\rho)\dot{\p} - \p \dot{\rho} \big) = \Delta \dot{\rho}. 
    \end{equation}
    Furthermore, the estimate \eqref{eq:dt f in L2H1} implies 
    \begin{equation}\label{eq:dt rho and p in L2H1}
   \begin{aligned}
       \big( \Vert \dot{\rho} &\Vert^2_{L^\infty(t,T;L^2(\Upsilon))} +\Vert \dot{\rho} \Vert^2_{L^2(t,T;H^1(\Omega))} \big) + \big( \Vert \dot{\p} \Vert^2_{L^\infty(t,T;L^2(\Omega))}+\Vert \dot{\p} \Vert^2_{L^2(t,T;H^1(\Omega))} \big) \\ 
       &\,\,\,\,\,\,\,\,\,\,\,\,\, \,\,\,\, \leq C\big(1+ \Vert f \Vert^2_{L^\infty((t,T)\times\Upsilon)}\big)^2\exp\big({T(1+ \Vert f \Vert^2_{L^\infty((t,T)\times\Upsilon))}}\big)\big(1+\Vert \dot{f}(t,\cdot) \Vert_{L^2(\Upsilon)}\big). 
   \end{aligned} 
\end{equation}
Note that, as per the proof of Corollary \ref{cor: gradients in L4 rho and p}, the Interpolation Lemma \ref{lem:dibenedetto classic} yields that, for a.e.~$t \in (0,T)$, 
\begin{equation}\label{eq:rho dot L4}
   \begin{aligned}
       \Vert \dot{\rho} &\Vert_{L^4((t,T)\times \Omega)} + \Vert \dot{\p} \Vert_{L^4((t,T)\times \Omega)} \\ 
       &\,\,\,\, \leq  \big( \Vert \dot{\rho} \Vert_{L^\infty(t,T;L^2(\Upsilon))} +\Vert \dot{\rho} \Vert_{L^2(t,T;H^1(\Omega))} \big) + \big( \Vert \dot{\p} \Vert_{L^\infty(t,T;L^2(\Omega))}+\Vert \dot{\p} \Vert_{L^2(t,T;H^1(\Omega))} \big). 
   \end{aligned} 
\end{equation}

\smallskip

\noindent 2. \textit{$H^2$-bound on $\dot{\rho}$}: The computation that follows is formal, as we do not know \emph{a priori} that $\Delta \dot{\rho}$ is square-integrable, however it is easily made rigorous by either the difference quotient technique used in the proof of Lemma \ref{lem:eqn for f dot} or the mollification method from the proof of Lemma \ref{lem:H2 bound for C1}. By testing equation \eqref{eq:eqn for dot rho} against $\Delta \dot{\rho}$, we obtain 
\begin{equation*}
    \begin{aligned}
        \frac{1}{2}\frac{\der}{\der t}\int_\Omega |\nabla \dot{\rho}|^2 \d \x + \int_\Omega |\Delta \dot{\rho}|^2 \d \x \leq & \int_\Omega |\Delta \dot{\rho}| (1-\rho)|\nabla \dot{\p}| \d \x + \int_\Omega |\Delta \dot{\rho}| |\dot{\p}||\nabla \rho | \d \x \\ 
        &+ \int_\Omega |\Delta \dot{\rho}| |\dot{\rho}| |\nabla \p| \d \x + \int_\Omega |\Delta \dot{\rho}| |\p| |\nabla \dot{\rho}| \d \x 
    \end{aligned}
\end{equation*}
Using Lemma \ref{lem:CZ periodic} and the H\"older inequality, as well as \eqref{eq:rho dot L4} and Corollary \ref{cor: gradients in L4 rho and p}, we obtain, for a.e.~$t \in (0,T)$, 
\begin{equation*}
   \begin{aligned}
       \Vert \nabla \dot{\rho} \Vert^2_{L^\infty(t,T;L^2(\Omega))} \!\! + \Vert \nabla \dot{\rho} \Vert^2_{L^2(t,T;H^1(\Omega))}\!\! \leq C\Big( &\Vert \nabla \dot{\rho}(t,\cdot) \Vert_{L^2(\Omega)}^2 \!\! + \Vert \nabla \dot{\rho} \Vert_{L^2((t,T)\times\Omega)}+ \Vert \nabla \dot{\p} \Vert_{L^2((t,T)\times\Omega)}^2  \\ 
       &+ \Vert \nabla \rho \Vert^4_{L^4((t,T)\times\Omega)} + \Vert \dot{\rho} \Vert^4_{L^4((t,T)\times\Omega)} \\ 
       &+ \Vert \nabla \p \Vert^4_{L^4((t,T)\times\Omega)} + \Vert \dot{\p} \Vert^4_{L^4((t,T)\times\Omega)} \Big) \\ 
       =:&(\phi^1(t))^2, 
   \end{aligned}
\end{equation*}
where the right-hand side is finite a.e.~and explodes as $t \to 0^+$. The Interpolation Lemma \ref{lem:dibenedetto classic} therefore yields $\Vert \nabla \dot{\rho} \Vert_{L^4((t,T)\times\Upsilon)} \leq C\phi^1(t)$ for a.e.~$t \in (0,T)$. Using also \eqref{eq:rho dot L4}, we deduce that, for a.e.~$t \in (0,T)$, 
\begin{equation*}
    \Vert \dot{\rho} \Vert_{L^\infty(t,T;L^2(\Omega))} + \Vert \dot{\rho} \Vert_{L^4(t,T;W^{1,4}(\Omega))} \leq C\phi^1(t), 
\end{equation*}
whence the Interpolation Lemma \ref{lem:dibenedetto classic} yields the result; up to a constant which we do not relabel. 
\end{proof}

We are now in a position to prove the first step of the regularity bootstrap. 

\begin{prop}[Boundedness of $\dot{f}$ away from Initial Time]\label{prop:first time deriv bounded by de giorgi}
    There exists a decreasing positive function $\psi^1:(0,T) \to (0,\infty)$ such that $\limsup_{t\to0^+}\psi^1(t) = +\infty$ and, for a.e.~$t \in (0,T)$, there holds $$\Vert \dot{f} \Vert_{L^\infty((t,T)\times\Upsilon)} \leq \psi^1(t).$$ 
\end{prop}

\begin{proof}
    We apply the De Giorgi method to \eqref{eq:eqn differentiated once in time}, which we rewrite as: 
    \begin{equation}\label{eq:eqn differentiated once in time for de giorgi}
    \de_t \fd +  \dv_{\bxi} \big(  U^1\fd + V^1 \big)
= \Delta_{\bxi} \fd, 
\end{equation} 
where $$U^1 := (1-\rho)\left(\begin{matrix}\e(\theta)\\ 0 \end{matrix}\right), \qquad V^1 := f\dot{\rho}\left(\begin{matrix}\e(\theta)\\0\end{matrix}\right).$$ As $f \in L^\infty((t,T)\times\Upsilon)$ for a.e.~$t\in(0,T)$, we conclude from Lemma \ref{lem:L8 for dot rho} that the drift terms satisfy: 
\begin{equation*}
    |U^1| \leq 1, \qquad V^1 \in L^{8}((t,T)\times \Upsilon) \quad \text{for a.e.~}t \in (0,T). 
\end{equation*}

\smallskip

    \noindent 1. \textit{Rescaling to the unit subcylinder}: Let $\delta \in (0,1)$, $(t_0, \bxi_0) \in (0,T)\times\mathbb{R}^3$, and $r$ satisfy the constraint: 
    \begin{equation}\label{eq:r smallness de giorgi on dot f}
        0 < r < \min\big\{1,\sqrt{t_0/4}\big\}. 
    \end{equation}
    For the purposes of what follows, we write 
    \begin{equation}
        \Vert \dot{f} \Vert_{t_0}^2 := \Vert \dot{f} \Vert^2_{L^\infty(t_0/2,T;L^2(\Upsilon))} + \Vert \nabla_{\bxi}\dot{f} \Vert^2_{L^2((t_0/2,T)\times\Upsilon)}. 
    \end{equation}
    Let $(t,\bxi) \in Q_r(t_0,\bxi_0)$ and define 
    \begin{equation*}
        \ell ({r,\delta}) := \delta^{\frac{1}{2}} \frac{r^{\frac{3}{2}}}{1+\Vert \dot{f} \Vert_{t_0} + \Vert V^1 \Vert_{L^{8}((t_0/2,T)\times \Upsilon)}}, 
    \end{equation*}
    as well as the rescaled functions $\dot{f}_r,U^1_r,V^1_r:Q_1 \to \mathbb{R}$ by 
     \begin{equation}\label{eq:rescaling ii}
        \begin{aligned}
            &\dot{f}_r(\tau,\bzeta) := \ell \dot{f}(t+r^2\tau,\bxi + r\bzeta), \\ 
            &U^1_r(\tau,\bzeta) := r U^1(t+r^2 \tau, \bxi + r \bzeta), \\ 
            &V^1_r(\tau,\bzeta) := r \ell V^1(t+r^2 \tau, \bxi + r \bzeta). 
        \end{aligned}
    \end{equation}
    
    Then, by arguing as per the proof of Lemma \ref{lem:rescaling i}, we have that $\dot{f}_r \in C([-1,0];L^2(B_1))\cap L^2(-1,0;H^1(B_1))$ is non-negative, $\partial_\tau\dot{f}_r \in L^2(-1,0;(H^1)'(B_1))$, $$|U^1_r| \leq 1 \quad \text{a.e.~in } Q_1, \qquad \esssup_{\tau \in [-1,0]}\int_{B_1} |\dot{f}_r(\tau)|^2 \d\bzeta + \int_{Q_1} |\nabla_{\bzeta}\dot{f}_r|^2 \d\bzeta \d \tau \leq \delta$$ and as per \eqref{eq:rescaled eqn int form} there holds, in the sense of distributions, 
    \begin{equation}\label{eq:rescaled eqn ii}
           \partial_\tau \dot{f}_r + \dv_{\bzeta}( U^1_r \dot{f}_r + V^1_r) = \Delta_{\bzeta} \dot{f}_r \qquad \text{in } Q_1. 
       \end{equation}
       Similarly, arguing as per the proof of Lemma \ref{lem:rescaling i} and using the restrictions on $\delta,r$, we obtain 
       \begin{equation*}
          \begin{aligned}
              \Vert V^1_r \Vert^8_{L^8(Q_1)} &\leq \frac{\delta^4 r^{7}}{\Vert V^1\Vert_{L^8((t_0/2,T)\times\Upsilon)}^8}\int_{Q_r} |V^1(t+\tau' , \bxi + \bzeta')|^8 \d\bzeta' \d\tau' \\ 
              &\leq \frac{1}{\Vert V^1\Vert_{L^8((t_0/2,T)\times\Upsilon)}^8}\int_{t_0/2}^T \int_\Upsilon |V^1(\bxi,t)|^8 \d \bxi \d t, 
          \end{aligned} 
       \end{equation*}
       whence there holds 
\begin{equation}\label{eq:bound for rescaled V1}
     \Vert V^1_r \Vert_{L^{8}(Q_1)} \leq 1. 
\end{equation}

\smallskip
       
    \noindent 2. \textit{Boundedness and compact exhaustion}: We are in a position to apply Lemmas \ref{lem:energy ineq weak sol prior to local boundedness} and \ref{lem:interior local boundedness weak sol} to the equation \eqref{eq:rescaling ii}. In turn, we find that $(\dot{f}_r)_+ \leq 1/2$ inside the subcylinder $Q_{1/2}$, from which we deduce: there exists a positive constant $C$, independent of $r,(t,\bxi)$, such that 
    \begin{equation}
       \Vert \dot{f} \Vert_{L^\infty(Q_r(t,\bxi))} \leq C\delta_*^{-\frac{1}{2}}(1+r^{-\frac{3}{2}}) \big( 1 + \Vert \dot{f} \Vert_{t_0} + \Vert V_1 \Vert_{L^8((t_0/2,T)\times\Upsilon)} \big). 
    \end{equation}
    We notice once again that the smallness requirement on $r$ in \eqref{eq:r smallness de giorgi on dot f} only depends on $t$, whence the exhaustion argument in the proof of Proposition \ref{thm:global boundedness away from initial time} may be repeated. We deduce that there exists a constant $C$ independent of $t$ such that, for a.e.~$t\in(0,T)$, there holds 
    \begin{equation*}
        \Vert \dot{f} \Vert_{L^\infty((t,T)\times\Upsilon)} \leq C (1+t^{-\frac{13}{4}} )\big( 1 + \Vert \dot{f} \Vert_{t} + \Vert V_1 \Vert_{L^8((t/2,T)\times\Upsilon)} \big) =: \psi^1(t), 
    \end{equation*}
    where the exponent $-13/4$ is obtained as in the proof of Proposition \ref{thm:global boundedness away from initial time}. The proof is complete.
\end{proof}

We conclude this subsection by noting that, by replicating the proofs of Lemmas \ref{lem:H2 bound for C1} and \ref{lem:eqn for f dot}, it is easy to see that, for a.e.~$t \in (0,T)$, there holds $f^{(2)} \in L^\infty(t,T;L^2(\Upsilon)) \cap L^2(t,T;H^1(\Upsilon))$ with $\partial_t f^{(2)} \in L^2(t,T;(H^1)'(\Upsilon))$, which satisfies in the weak sense 
\begin{equation*}
    \partial_t f^{(2)} + \dv \big( (1-\rho)f^{(2)} - 2 f^{(1)} \rho^{(1)} - f \rho^{(2)} )\e(\theta)\big) = \Delta_{\bxi} f^{(2)}. 
\end{equation*}

\subsection{Proofs of Propositions \ref{prop:boundedness of all time derivs} and \ref{prop:sobolev of all time derivs}}\label{sec:smooth away from initial time}

Proposition \ref{prop:boundedness of all time derivs} follows immediately by applying the next lemma inductively. 

\begin{lemma}
    Let $f$ be a non-negative weak solution of \eqref{eq:main eqn}. Let $n \geq 2$ be an integer. Assume that for $j \in \{0,\dots,n-1\}$ there exist decreasing positive functions $\psi^j:(0,T)\to(0,\infty)$ such that $\limsup_{t\to0^+}\psi^j(t) = +\infty$ and, for a.e.~$t \in (0,T)$, there holds 
    \begin{equation}\label{eq:de giorgi up to n-1}
    \Vert f^{(j)} \Vert_{L^\infty((t,T)\times\Upsilon)} \leq \psi^j(t), 
    \end{equation}
    and there exists positive decreasing functions $\Phi^j:(0,T) \to (0,\infty)$ such that 
    \begin{equation}\label{eq:general jth deriv parabolic bound}
        \begin{aligned}
            \Vert f^{(j)} \Vert_{L^\infty(t,T;H^1(\Upsilon))} + \Vert f^{(j)} \Vert_{L^2(t,T;H^2(\Upsilon))} \leq \Phi^j(t) \quad \text{a.e.~}t\in(0,T). 
        \end{aligned}
    \end{equation}
    Assume also that there exists a positive decreasing function $\varphi^n:(0,T) \to (0,\infty)$ such that 
        \begin{equation}\label{eq:general nth deriv parabolic bound}
        \begin{aligned}
            \Vert f^{(n)} \Vert_{L^\infty(t,T;L^2(\Upsilon))} + \Vert f^{(n)} \Vert_{L^2(t,T;H^1(\Upsilon))} \leq \varphi^n(t) \quad \text{a.e.~}t\in(0,T), 
        \end{aligned}
    \end{equation}
    and that there holds in the weak sense, with $\partial_t f^{(n)} \in L^2(0,T;(H^1)'(\Upsilon))$, 
    \begin{equation}\label{eq:general nth time deriv eqn}
    \de_t \ff{n} +  \dv  \left[
    \left( (1- \rho) \ff{n}
    + \sum_{k=1}^{n-1} \binom{n}{k} \rr{n-k} \ff{k} 
    -\rr{n} f
    \right) \,\e(\theta)
    \right]
= \Delta_{\bxi} \ff{n}. 
\end{equation}
Then, there exists a decreasing positive function $\psi^{n}:(0,T)\to(0,\infty)$ with $\limsup_{t\to0^+}\psi^{n}(t)=+\infty$ and, for a.e.~$t\in(0,T)$, there holds 
\begin{equation}\label{eq:fn bounded de giorgi} 
\Vert f^{(n)} \Vert_{L^\infty((t,T)\times\Upsilon)} \leq \psi^{n}(t).
\end{equation}
Furthermore, $f^{(n+1)} \in L^\infty(t,T;L^2(\Upsilon)) \cap L^2(t,T;H^1(\Upsilon))$ and $\partial_t f^{(n+1)} \in L^2(0,T;(H^1)'(\Upsilon))$ satisfies 
    \begin{equation}\label{eq:general n+1th time deriv eqn}
    \de_t \ff{n+1}\!  +  \dv \!\! \left[
    \left( (1- \rho) \ff{n+1}
    + \sum_{k=1}^{n} \binom{n+1}{k} \rr{n+1-k} \ff{k} 
    -\rr{n+1} f
    \right) \! \e(\theta)
    \right]
\! \! = \Delta_{\bxi} \ff{n+1}, 
\end{equation}
and there exists a positive decreasing function $\varphi^{n+1}:(0,T) \to (0,\infty)$ such that 
\begin{equation}\label{eq:general n+1th parabolic bound}
\Vert f^{(n+1)} \Vert_{L^\infty(t,T;L^2(\Upsilon))} + \Vert f^{(n+1)} \Vert_{L^2(t,T;H^1(\Upsilon))} \leq \varphi^{n+1}(t) \quad \text{a.e.~}t\in(0,T), 
\end{equation}
as well as a positive decreasing function $\Phi_n:(0,T)\to(0,\infty)$ such that 
\begin{equation}\label{eq:H2 parabolic at n level}
      \Vert f^{(n)} \Vert_{L^\infty(t,T;H^1(\Upsilon))} + \Vert f^{(n)} \Vert_{L^2(t,T;H^2(\Upsilon))} \leq \Phi^n(t) \quad \text{a.e.~}t\in(0,T). 
\end{equation}
\end{lemma}

\begin{proof}

Throughout this proof, we denote the binomial constants by $\binom{n}{k} =: C^n_k.$

\smallskip 
\noindent 1. \textit{Improved integrability for $\rho^{(n)}$}: By integrating \eqref{eq:general nth time deriv eqn} with respect to the angle variable, we see that there holds in the weak sense 
\begin{equation*}
    \partial_t \rho^{(n)} + \dv\Big( (1-\rho) \p^{(n)} + \sum_{k=1}^{n-1} C^n_k \rho^{(n-k)} \p^{(k)} - \p\rho^{(n)}  \Big) = \Delta \rho^{(n)}. 
\end{equation*}
The Interpolation Lemma \ref{lem:dibenedetto classic} and estimates \eqref{eq:general jth deriv parabolic bound} imply that, for all $j \in \{0,\dots,n-1\}$, there holds 
\begin{equation}\label{eq:L4 for multiple time derivs}
   \begin{aligned} 
   \Vert \rho^{(j)} & \Vert_{L^4((t,T)\times\Omega)} + \Vert \p^{(j)} \Vert_{L^4((t,T)\times\Omega)} \\ 
    &+\Vert \nabla \rho^{(j)} \Vert_{L^4((t,T)\times\Omega)} + \Vert \nabla \p^{(j)} \Vert_{L^4((t,T)\times\Omega)} \leq C\Phi^j(t) \quad \text{a.e.~}t\in(0,T), 
   \end{aligned}
\end{equation}
as well as $\Vert \rho^{(n)} \Vert_{L^4((t,T)\times\Omega)} + \Vert \p^{(n)} \Vert_{L^4((t,T)\times\Omega)} \leq C\varphi^n(t)$. 

We replicate the second derivative estimate from the proof of Lemma \ref{lem:L8 for dot rho}; we write down formal estimates for clarity of presentation, which can easily be made rigorous by means of difference quotients or mollification. We test \eqref{eq:general nth time deriv eqn} against $\Delta_{\bxi} \rho^{(n)}$ and obtain 
\begin{equation*}
    \begin{aligned}
        \frac{1}{2}\frac{\der}{\der t} &\int_\Omega |\nabla \rho^{(n)}|^2 \d x + \int_\Omega |\Delta \rho^{(n)}|^2 \d x \\ 
        \leq & \int_\Omega |\Delta \rho^{(n)}| |\nabla \p^{(n)}| \d x + \int_\Omega |\Delta \rho^{(n)}| |\p^{(n)}| |\nabla \rho| \d x + C\sum_{k=1}^{n-1}\int_\Omega |\Delta \rho^{(n)}| |\rho^{(n-k)}| |\nabla \p^{(k)}| \d x \\ 
        &+ C\sum_{k=1}^{n-1}\int_\Omega |\Delta \rho^{(n)}| |\p^{(k)}| |\nabla \rho^{(n-k)}| \d x + \int_\Omega |\Delta \rho^{(n)}| |\nabla \rho^{(n)}| \d x + \int_\Omega |\Delta \rho^{(n)}| |\nabla \p | |\rho^{(n)}| \d x. 
    \end{aligned}
\end{equation*}
By applying the H\"older and Young inequalities, we get 
\begin{equation*}
    \begin{aligned}
        \Vert \nabla&\rho^{(n)} \Vert^2_{L^\infty(t,T;L^2(\Omega))} + \Vert \Delta \rho^{(n)} \Vert^2_{L^2((t,T)\times\Omega)} \\ 
        \leq & \, C \bigg( \Vert\nabla\rho^{(n)}(t,\cdot) \Vert^2_{L^2(\Omega)} + \Vert \nabla \rho^{(n)} \Vert^2_{L^2((t,T)\times\Omega)} + \Vert \nabla \p^{(n)}\Vert_{L^2((t,T)\times\Omega)}^2 + \Vert \rho^{(n)} \Vert^4_{L^4((t,T)\times\Omega)}\\ 
        &+ \Vert\p^{(n)}\Vert^4_{L^4((t,T)\times\Omega)} + \sum_{k=0}^{n-1}\Big( \Vert \rho^{(k)} \Vert_{L^4(t,T;W^{1,4}(\Omega))}^4 + \Vert \p^{(k)} \Vert_{L^4(t,T;W^{1,4}(\Omega))}^4 \Big) \bigg) \\ 
        \leq& \, C\Big( \Vert\nabla\rho^{(n)}(t,\cdot) \Vert^2_{L^2(\Omega)} +  (\varphi^n(t))^2 + \sum_{k=0}^{n-1} (\Phi^k(t))^4 \Big) \\ 
        =:& \, \hat{\Psi}^n(t); 
    \end{aligned}
\end{equation*}
the above is bounded by virtue of the boundedness of $\varphi^n$ (see \eqref{eq:general nth deriv parabolic bound}), $\{\Phi^k\}_{k=0}^{n-1}$ and $\Vert\nabla\rho^{(n)}(t,\cdot) \Vert_{L^2(\Omega)}$ being finite for a.e.~$t\in(0,T)$ using Markov's inequality and \eqref{eq:general nth deriv parabolic bound}.

It follows from Lemma \ref{lem:CZ periodic} and the Interpolation Lemma \ref{lem:dibenedetto classic} that $\Vert \nabla \rho^{(n)} \Vert_{L^4((t,T)\times\Omega)} \leq C\hat{\Psi}^n(t)$. Arguing as per the proof of Lemma \ref{lem:L8 for dot rho}, we deduce that there exists a positive decreasing function $\phi^n$ such that 
\begin{equation*}
    \Vert \rho^{(n)} \Vert_{L^\infty(t,T;L^2(\Omega))} + \Vert \rho^{(n)} \Vert_{L^4(t,T;W^{1,4}(\Omega))} \leq \phi^n(t), 
\end{equation*}
whence the Interpolation Lemma yields 
\begin{equation}\label{eq:L8 for nth time deriv rho}
    \Vert \rho^{(n)} \Vert_{L^8((t,T)\times\Omega)} \leq C \phi^n(t). 
\end{equation}

\smallskip

\noindent 2. \textit{De Giorgi method for $f^{(n+1)}$}: We define the drift term 
\begin{equation*}
    V^n := \sum_{k=1}^{n-1} C^n_k \rho^{(n-k)} \p^{(k)} - \p\rho^{(n)}. 
\end{equation*}
We deduce directly from \eqref{eq:de giorgi up to n-1}, \eqref{eq:L4 for multiple time derivs}, and \eqref{eq:L8 for nth time deriv rho} that there holds 
\begin{equation*}
    V^n \in L^8((t,T)\times\Omega) \quad \text{a.e.~}t \in (0,T). 
\end{equation*}
By following the proof of Proposition \ref{prop:first time deriv bounded by de giorgi} to the letter, we apply the De Giorgi method to the equation \eqref{eq:general nth time deriv eqn} and obtain the estimate \eqref{eq:fn bounded de giorgi}. 

\smallskip

\noindent 3. \textit{Second derivative estimate}:  We replicate the argument of Lemma \ref{lem:H2 bound for C1} for the equation \eqref{eq:general nth time deriv eqn}; again the estimates are formal for clarity of presentation, and can easily be made rigorous by means of difference quotients or mollification. We test \eqref{eq:general nth time deriv eqn} against $\Delta_{\bxi} f^{(n)}$ and obtain 
\begin{equation*}
    \begin{aligned}
        \frac{1}{2}\frac{\der}{\der t}\int_\Upsilon | \nabla_{\bxi} & f^{(n)}|^2 \d \bxi + \int_\Upsilon |\Delta_{\bxi} f^{(n)}|^2 \d \bxi \\ 
        \leq & \int_\Upsilon |\Delta_{\bxi} f^{(n)}| |\nabla_{\bxi} f^{(n)}|  \d \bxi + \int_\Upsilon |\Delta_{\bxi}f^{(n)}| |f^{(n)}| |\nabla \rho| \d \bxi \\ 
        &+ C\sum_{k=1}^{n-1} \int_{\Upsilon} |\Delta_{\bxi} f^{(n)}| |\rho^{(n-k)}| |\nabla f^{(k)}| \d \bxi + C\sum_{k=1}^{n-1} \int_{\Upsilon} |\Delta_{\bxi} f^{(n)}| |\nabla \rho^{(n-k)}| |f^{(k)}| \d \bxi \\ 
        &+ \int_\Upsilon |\Delta_{\bxi} f^{(n)}| |\rho^{(n)}| |\nabla f| \d \bxi + \int_\Upsilon |\Delta_{\bxi} f^{(n)}| |\nabla \rho^{(n)}| |f| \d \bxi. 
    \end{aligned}
\end{equation*}
Using the Young and Jensen inequalities as well as the assumption \eqref{eq:de giorgi up to n-1}, the above implies, for a.e.~$t \in (0,T)$, 
\begin{equation*}
    \begin{aligned}
       \frac{\der}{\der t} & \Vert \nabla_{\bxi} f^{(n)}(t)\Vert_{L^2(\Upsilon)}^2 + \Vert \Delta_{\bxi} f^{(n)}(t)\Vert_{L^2(\Upsilon)}^2 \\ 
        &\leq C\bigg( \!\Vert \nabla_{\bxi}f^{(n)} \Vert^2_{L^2(\Upsilon)}\! + \Vert f^{(n)} \Vert_{L^\infty((t,T)\times\Upsilon)}^2 \Vert \nabla f\Vert^2_{L^2(\Upsilon)} \!\! +\!\! \sum_{k=1}^{n-1}\Vert f^{(k)} \Vert_{L^\infty((t,T)\times\Upsilon)}^2 \Vert \nabla f^{(n-k)} \Vert^2_{L^2(\Upsilon)}  \! \bigg) \\ 
        &\leq C\bigg( \Vert \nabla_{\bxi}f^{(n)} (t)\Vert^2_{L^2(\Upsilon)} + (\psi^n(t))^2 (\Phi^0(t))^2 + \sum_{k=1}^n (\psi^k(t))^2 (\Phi^{n-k}(t))^2 \bigg), 
    \end{aligned}
\end{equation*}
and an application of Gr\"onwall's Lemma yields the estimate \eqref{eq:H2 parabolic at n level}. Returning to the equation, we deduce, as per the estimate \eqref{eq:using the equation} that $\Vert \partial_t f^{(n)} \Vert_{L^2((t,T)\times\Upsilon)}$ is finite for a.e.~$t\in(0,T)$. 

\smallskip

\noindent 4. \textit{Equation differentiated in time}: We now differentiate \eqref{eq:general nth time deriv eqn} in time; the new coefficients $\{C^{n+1}_k \}_{k=1}^n$ are determined by the product rule. For clarity of presentation, we justify this step by performing formally the classical parabolic estimate on \eqref{eq:general n+1th time deriv eqn}, which can be done rigorously by means of difference quotients as per the proof of Lemma \ref{lem:eqn for f dot}. We test \eqref{eq:general n+1th time deriv eqn} with $f^{(n+1)}$ and obtain 
\begin{equation*}
    \begin{aligned}
        \frac{1}{2}&\frac{\der}{\der t} \int_\Upsilon |f^{(n+1)}|^2 \d \bxi + \int_\Upsilon |\nabla_{\bxi} f^{(n+1)}|^2 \d \bxi \\ 
        &\leq \int_\Upsilon |\nabla f^{(n+1)}||f^{(n+1)} | \d \bxi + C\sum_{k=1}^\infty |\nabla f^{(n+1)}| |\rho^{(n+1-k)}||f^{(k)}| \d \bxi + \int_\Upsilon |\nabla f^{(n+1)}| |\rho^{(n+1)}| |f| \d \bxi, 
    \end{aligned}
\end{equation*}
    whence the Young and Jensen inequalities yield, using also \eqref{eq:de giorgi up to n-1}, 
    \begin{equation*}
        \begin{aligned}
       \frac{\der}{\der t} &\Vert f^{(n+1)}(t)\Vert_{L^2(\Upsilon)}^2 \Vert \nabla_{\bxi} f^{(n+1)}(t)\Vert^2_{L^2(\Upsilon)} \\ 
        \leq & \, C\Big((1+ \Vert f \Vert_{L^\infty((t,T)\times\Upsilon)})\Vert f^{(n+1)}(t) \Vert^2_{L^2(\Upsilon)} + \sum_{k=1}^{n} \Vert f^{(n+1-k)}\Vert_{L^\infty((t,T)\times\Upsilon)} \Vert f^{(k)} \Vert_{L^\infty((t,T)\times\Upsilon)} \Big) \\ 
        \leq & \, C\Big( (1+\psi^0(t)) \Vert f^{(n+1)}(t) \Vert^2_{L^2(\Upsilon)} +  \sum_{k=1}^{n} \psi^{n+1-k}(t) \psi^k(t)\Big), 
    \end{aligned}
\end{equation*}
whence the estimate \eqref{eq:general n+1th parabolic bound} follows from Gr\"onwall's Lemma and the weak formulation \eqref{eq:general n+1th time deriv eqn} is justified. 
\end{proof}

\begin{proof}[Proof of Proposition \ref{prop:boundedness of all time derivs}]

    As previously mentioned, Proposition \ref{prop:boundedness of all time derivs} follows from the previous lemma by applying it inductively.     
\end{proof}

We shall now employ Proposition \ref{prop:boundedness of all time derivs} and the time-differentiated equations to prove Proposition \ref{prop:sobolev of all time derivs}.

\begin{proof}[Proof of Proposition \ref{prop:sobolev of all time derivs}]
    Estimate \eqref{eq:H2 parabolic at n level}, Jensen's inequality, and the Interpolation Lemma imply that for a.e.~$t\in(0,T)$ 
    \begin{equation*}
        \Vert \nabla_{\bxi} f^{(n)} \Vert_{L^{\frac{10}{3}}((t,T)\times\Upsilon)} + \Vert \nabla \rho^{(n)} \Vert_{L^4((t,T)\times\Omega)} \leq \Phi^n(t), 
    \end{equation*}
  and the above holds for all $n$. It then follows from the equation \eqref{eq:general nth time deriv eqn} and Jensen's inequality that 
    \begin{equation*}
        \Vert \Delta_{\bxi} f^{(n)} \Vert_{L^{\frac{10}{3}}((t,T)\times\Upsilon)} \leq C\Big( \Vert f^{(n+1)} \Vert_{L^\infty((t,T)\times\Upsilon)} + \sum_{k=0}^{n} \Vert f^{(k)} \Vert_{L^\infty((t,T)\times\Upsilon)}\Vert \nabla f^{(n-k)} \Vert_{L^{\frac{10}{3}}((t,T)\times\Upsilon)}\Big). 
    \end{equation*}
    The above and Lemma \ref{lem:CZ periodic} imply an equivalent bound on the full Hessian $\Vert \nabla^2_{\bxi} f^{(n)} \Vert_{L^{10/3}((t,T)\times\Upsilon)}$. Meanwhile, $\partial^2_t f^{(n)} = f^{(n+2)} \in L^\infty((t,T)\times\Upsilon)$, and the result follows from Minkowski's inequality. 
\end{proof}

\section{Uniqueness and Regularity of Very Weak Solutions}\label{sec:very weak}

We begin this section with a uniqueness result \emph{\`a la} Michel Pierre for very weak solutions of \eqref{eq:main eqn}; this result will then be used show that very weak solutions coincide with weak solutions away from the initial time, whence they are endowed with the same regularity properties derived in \S \ref{sec:higher reg}.

\begin{lemma}[Uniqueness of Very Weak Solutions away from Initial Time]\label{lem:very weak uniqueness}
    Let $f$ and $g$ be very weak solutions of \eqref{eq:main eqn}. Suppose that, for some $t_0 \in (0,T)$, there holds $f(t_0,\cdot)=g(t_0,\cdot)$ in $L^2(\Upsilon)$. Then, $f = g$ in $L^2((t_0,T)\times\Upsilon)$. 
\end{lemma}

Before we proceed to the proof of this result, we remark that a standard argument shows that the weak and very weak formulations of Definition \ref{def:concept of solution} can be rewritten without the duality product of the time-derivative, provided we employ test functions that vanish along $\{t=0\}$ and $\{t=T\}$. Indeed, provided $\varphi \in C^\infty(\Upsilon_T)$ is such that $\varphi(0,\cdot)=\varphi(T,\cdot)=0$, then there holds 
\begin{equation*}
    \langle \partial_t f , \varphi \rangle = - \int_{\Upsilon_T} f \partial_t \varphi \d \bxi \d t. 
\end{equation*}
We use this fact in the proof that follows. 

\begin{proof}
Recall from \cite[Proof of Theorem 2.6]{bbes} that $f,g \in L^2(\Upsilon_T)$ and the integrals $$ \rho_f(t,x) = \int_0^{2\pi} f(t,x,\theta) \d \theta, \qquad \rho_g(t,x) = \int_0^{2\pi} g(t,x,\theta) \d \theta$$ are well-defined and $0\leq \rho_f,\rho_g \leq 1$ a.e.~in $\Omega_T$. Without loss of generality, we translate the problem in time $t \mapsto t-t_0$ such that we may assume $t_0=0$. 

    Let $w := f - g \in L^2(\Upsilon_T)$ and note that the assumptions of the lemma imply $w(0,\cdot)=0$. Hence, for any smooth function $\psi$ periodic in $x,\theta$ satisfying $\psi(T,\cdot) =0$, there holds 
    \begin{equation}\label{eq:test difference for uniqueness}
       \begin{aligned}
            \int_{\Upsilon_T}  w \partial_t \psi \d \bxi \d t = &\int_{\Upsilon_T} \nabla w \cdot \nabla \psi \d \bxi \d t - \int_{\Upsilon_T} w \partial^2_\theta \psi \d\bxi \d t -\int_{\Upsilon_T} w \e(\theta) \cdot \nabla \psi \d \bxi \d t \\ 
           & + \int_{\Upsilon_T}  \rho_f w \e(\theta)  \cdot  \nabla \psi \d \bxi \d t + \int_{\Upsilon_T}w \bigg(\int_0^{2\pi} g \e(\theta') \cdot \nabla \psi \d \theta'\bigg) \d \bxi \d t, 
       \end{aligned}
    \end{equation}
where we used the Fubini--Tonelli Theorem to rewrite the final term on the right-hand side, \textit{i.e.}, 
\begin{equation*}
   \begin{aligned} \int_{\Upsilon_T} \!\!\!\bigg( \int_0^{2\pi} \!\!\! w(t,x,\theta') \d \theta' \bigg) g(t,x,\theta) & \e(\theta) \cdot \nabla \psi(t,x,\theta) \d \bxi \d t \\ 
   &= \int_{\Upsilon_T}\!\!\!\!\! w \bigg(\int_0^{2\pi}\!\!\!\! g(t,x,\theta') \e(\theta') \cdot \! \nabla \psi(t,x,\theta') \d \theta' \bigg) \d \bxi \d t. 
\end{aligned}
\end{equation*}
A standard density argument implies that \eqref{eq:test difference for uniqueness} holds for all $\psi \in L^2(0,T;H^2(\Upsilon))$ periodic in $x,\theta$ satisfying the final-time condition $\psi(T,\cdot)=0$. 

In turn, for $\zeta \in C^\infty_c(\Upsilon_T)$ an arbitrary test function periodic in $x,\theta$, we define $\phi$ to be a strong solution of the following strongly parabolic linear (formal) dual equation: 
    \begin{equation}\label{eq:dual eqn}
        \left\lbrace\begin{aligned}
            &\partial_t \phi - \Delta_{\bxi} \phi - (1-\rho_f) \e(\theta) \cdot \nabla \phi + \int_0^{2\pi} g \e(\theta')\cdot \nabla \phi \d \theta' = -\zeta, \\ 
            & \phi(0,\cdot) = 0, 
        \end{aligned}\right. 
    \end{equation}
with periodic boundary conditions in $x,\theta$, and consider $\psi(t,\bxi) := \phi(T-t,\bxi)$, which satisfies 
    \begin{equation}\label{eq:dual eqn actual one we use}
        \left\lbrace\begin{aligned}
            &\partial_t \psi +\Delta_{\bxi} \psi + (1-\rho_f) \e(\theta) \cdot \nabla \psi - \int_0^{2\pi} g \e(\theta')\cdot \nabla \psi \d \theta' = \zeta, \\ 
            &\psi(T,\cdot) = 0. 
        \end{aligned}\right. 
    \end{equation}
    The standard linear theory implies the existence and uniqueness of $\phi \in C([0,T];L^2_\per(\Upsilon)) \cap L^2(0,T;H^1_\per(\Upsilon))$, and regularity arguments akin to those used in the proof of Lemma \ref{lem:H2 bound for C1} imply $\phi \in L^2(0,T;H^2_\per(\Upsilon))$. Hence $\psi$ is an admissible test function to insert into \eqref{eq:test difference for uniqueness}, whence integrating \eqref{eq:dual eqn actual one we use} against $w$ and substituting for the term $\int_{\Upsilon_T} w \partial_t \psi \d \bxi \d t$ using \eqref{eq:test difference for uniqueness} yields 
    \begin{equation*}
        \int_{\Upsilon_T} w \zeta \d\bxi \d t = 0. 
    \end{equation*}
This procedure may be repeated for all $\zeta \in C^\infty_c(\Upsilon_T)$, whence the previous equality holds for all such $\zeta $. It follows that $w \equiv 0$ a.e.~in $\Upsilon_T$, which concludes the proof of the lemma. 
\end{proof}

In turn, we are ready to give the proofs of the main results for very weak solutions. 

\begin{proof}[Proof of Theorems \ref{thm:weak strong uniqueness} and \ref{thm:reg very weak}]
     Let $f$ be a very weak solution. There holds $f \in L^2(\Upsilon_T)$, whence, by Markov's inequality, for a.e.~$t_0 \in (0,T)$ there holds $f(t_0,\cdot) \in L^2(\Upsilon)$. It follows that $f(t_0,\cdot)$ may be used as initial data to produce a weak solution for a.e.~$t \in (t_0,T)$, which we denote by $g$. It follows from Lemma \ref{lem:very weak uniqueness} that $f \equiv g$ in $(t_0,T)$ in the sense of $L^2((t_0,T)\times\Upsilon)$, for a.e.~$t_0 \in (0,T)$; in other words, we may select $g$ as the \emph{precise representative} of $f$ in $(t_0,T)\times \Upsilon$. In turn, since $f$ is regular enough to satisfy the weak formulation of \eqref{eq:main eqn}, the conclusion of Theorem \ref{thm:weak strong uniqueness} follows. Furthermore, by Theorem \ref{thm:smooth away from initial}, there holds $g \in C^\infty((t,T)\times\mathbb{R}^3)$ a.e.~$t \in (0,T)$, whence $f$ satisfies the assertion of Theorem \ref{thm:reg very weak}. 
\end{proof}

\section{Stationary States}\label{sec:harris}

Our objective in this final section is to prove regularity results for the stationary equation \eqref{eq:equilibrium eqn} associated to \eqref{eq:main eqn} and the convergence of the time-dependent solutions to stationary solutions. We begin with the former, in \S \ref{subsec:reg stationary}, and then consider the latter in \S \ref{subsec:convergence}.

\subsection{Regularity of Stationary Solutions}\label{subsec:reg stationary}

Note that the existence of stationary solutions is trivial; by formally rewriting the drift term $\dv((1-\rho)f \e(\theta)) = \e(\theta) \cdot \nabla((1-\rho)f)$, all constant solutions $f_\infty \in \mathbb{R}$ satisfy the stationary equation \eqref{eq:equilibrium eqn}, \textit{i.e.}, 
\begin{equation*}
    \Pe \, \dv\big( (1-\rho_\infty)f_\infty \e(\theta) \big) = D_e \Delta f_\infty + \partial^2_\theta f_\infty 
\end{equation*}
with periodic boundary conditions.

We now prove Theorem \ref{thm:smooth stationary states}. To do so, we apply a bootstrapping argument similar to the one in \S \ref{sec:higher reg} and thereby show smoothness of the stationary states. Without loss of generality we restrict this part of our analysis to the case $\Pe=D_e=1$.

    \begin{proof}[Proof of Theorem \ref{thm:smooth stationary states}] 

    The proof is divided into several steps. 

    \smallskip 
    \noindent 1. \textit{$H^2$-type bound on $f_\infty$}: Using the classical Moser's iteration method in the elliptic context (\textit{cf.}~\S \ref{sec:first de giorgi lemma} or \cite[\S 2]{vasseur}), it holds that any weak solution $f_\infty \in H^1_\per(\Upsilon)$ of \eqref{eq:equilibrium eqn} with $0 \leq \rho_\infty \leq 1$ belongs to $L^\infty$; we omit the details for concision. Then, by testing the equation against $\Delta_{\bxi} f_\infty$ (which may be performed rigorously by means of difference quotients as in \S \ref{sec:higher reg}), we obtain using Young's inequality the second derivative estimate 
    \begin{equation*}
        \int_\Upsilon |\Delta_{\bxi}f_\infty |^2 \d \bxi \leq C \big( 1 + \Vert f_\infty \Vert_{L^\infty(\Upsilon)}^2\big) \int_\Upsilon |\nabla f_\infty |^2 \d \bxi, 
    \end{equation*}
     for some positive constant $C=C(\Upsilon)$. It follows from Lemma \ref{lem:CZ periodic} that $\Delta_{\bxi} f_\infty \in L^2(\Upsilon)$ implies $f_\infty \in W^{2,2}(\Upsilon)$ and thus $\rho_\infty \in W^{2,2}(\Omega)$. The Sobolev Embedding Theorem therefore implies $\nabla f_\infty \in L^6(\Upsilon)$ and $\nabla \rho_\infty \in BMO(\Omega)$, whence $\nabla \rho_\infty \in L^p(\Omega)$ for all $p \in [1,\infty)$. 

     \smallskip
     
     \noindent 2. \textit{Equation for $\rho_\infty$}: By integrating the equation with respect to the angle variable, we recover the equation for $\rho_\infty$, which reads 
     \begin{equation*}
         \dv\big( (1-\rho_\infty) \p_\infty \big) = \Delta \rho_\infty, 
     \end{equation*}
    where $\p_\infty = \int_0^{2\pi} f_\infty \e(\theta) \d \theta \in L^\infty(\Omega)$. By virtue of $0 \leq \rho_\infty \leq 1$, we obtain 
    \begin{equation}\label{eq:deducing Delta rho}
        \Delta \rho_\infty = (1-\rho_\infty) \underbrace{\dv \, \p_\infty}_{\in L^6} - \nabla \rho_\infty \cdot \p_\infty, 
    \end{equation}
    and we deduce from Lemma \ref{lem:CZ periodic} that $\rho_\infty \in W^{2,6}(\Omega)$ and Morrey's Embedding implies $\nabla\rho_\infty \in W^{1,6}(\Omega) \subset C^{0,2/3}(\Omega)$; in particular, we have $\rho_\infty,\nabla \rho_\infty \in L^\infty(\Omega)$. Returning to \eqref{eq:equilibrium eqn}, we get 
    \begin{equation}\label{eq:deducing Delta f}
        \Delta_{\bxi}f_\infty = (1-\rho_\infty) \underbrace{\nabla f_\infty}_{\in L^6} \cdot \e(\theta) - \underbrace{f_\infty \nabla \rho_\infty \cdot \e(\theta)}_{\in L^\infty}, 
    \end{equation}
    whence, again by Lemma \ref{lem:CZ periodic} and Morrey's Embedding, we get $ \nabla f_\infty \in W^{1,6}(\Upsilon) \subset C^{0,1/2}(\Upsilon)$; thus $f_\infty,\nabla f_\infty \in L^\infty(\Upsilon)$. By returning to \eqref{eq:deducing Delta rho}--\eqref{eq:deducing Delta f}, we get $\Delta_{\bxi} f_\infty \in L^\infty(\Upsilon)$ and $\Delta \rho_\infty \in L^\infty(\Omega)$; in fact, both of these quantities are H\"older continuous. By Lemma \ref{lem:CZ periodic}, we obtain 
    \begin{equation}\label{eq:W2p bound} 
    f_\infty \in W^{2,p}(\Upsilon) \quad \text{for all } p \in [1,\infty). 
    \end{equation}

\smallskip
    
  \noindent 3. \textit{Higher derivatives}: Next, one may use difference quotients with respect to the variable $\bxi$ to make rigorous the following formal computation: differentiating \eqref{eq:equilibrium eqn} with respect to $x$ and $\theta$, respectively, gives 
    \begin{equation*}
    \begin{aligned}
       & \dv\Big( \big( (1-\rho_\infty)\nabla f_\infty - f_\infty \nabla \rho_\infty \big)\otimes\e(\theta) \Big) = \Delta_{\bxi} \nabla f_\infty \\ 
       & \dv\big( (1-\rho_\infty)\partial_\theta f_\infty \e(\theta) +  (1-\rho_\infty) f_\infty \e'(\theta) \big) = \Delta_{\bxi} \partial_\theta f_\infty. 
        \end{aligned}
    \end{equation*}
    In turn, taking the $L^p$ norm of both sides and estimating directly using \eqref{eq:W2p bound} yields $\Delta_{\bxi} \nabla_{\bxi} f_\infty \in L^p(\Upsilon)$ for all $p \in [1,\infty)$. It follows from Lemma \ref{lem:CZ periodic} that $\nabla^3_{\bxi} f \in L^p(\Upsilon)$, \textit{i.e.}, 
    $$ \nabla^2_{\bxi} f_\infty \in W^{1,p}(\Upsilon) \subset C^{0,\gamma}(\Upsilon)$$ 
    for some $\gamma \in (0,1)$, by Morrey's Embedding. One may iterate this procedure indefinitely to deduce that $\nabla^k_{\bxi} f$ is continuous for all $k\in\mathbb{N}$, whence we deduce $f$ is smooth; we skip the details as they are analogous to those in \S \ref{sec:higher reg}. Note that the aforementioned embeddings may be taken with respect to the larger set $\Upsilon' = (-2\pi,4\pi)^3$ using the periodicity of $f_\infty$, whence $\Upsilon$ is comprised as an interior set and thus all embeddings also hold on the closure $\overline{\Upsilon}$, \textit{i.e.}, there exist $\gamma_k \in (0,1)$ such that $\nabla^k_{\bxi} f_\infty \in C^{0,\gamma_k}(\overline{\Upsilon})$ for all $k\in\mathbb{N}$. 
\end{proof}

The uniqueness of stationary solutions is a more delicate issue. Nevertheless, for small P\'eclet number, we have the following result regarding a natural linearisation of \eqref{eq:equilibrium eqn}. Note that the stationary function $\rho_\infty$ is known \emph{a priori}, since the uniform-in-time estimate $0 \leq \rho(t,\cdot) \leq 1$ implies the weak-* subsequential convergence of the sequence $\{ \rho(t,\cdot) \}_{t>0}$ to $\rho_\infty$ also satisfying $0 \leq \rho_\infty \leq 1$ a.e.

\begin{lemma}
    Let $m \geq 0$ and $\rho_\infty \in L^\infty_\per(\Omega)$ be given, and assume that $0 \leq \rho_\infty \leq 1$ a.e. Provided there holds $$|\Pe| < \frac{\min\{D_e,1\}}{C_P},$$
    where $C_P$ is the Poincar\'e constant on $\Upsilon$, then the solution of the linear elliptic equation 
    \begin{equation}\label{eq:linear elliptic eqn}
        \Pe \, \dv((1-\rho_\infty) f_\infty \e(\theta)) = D_e \Delta f_\infty + \partial^2_\theta f_\infty 
    \end{equation}
    is unique within the class $$\mathcal{C} = \Big\{f_\infty \in H^1_\per(\Upsilon): \, f_\infty \geq 0 \text{ a.e.}, \, \int_\Upsilon f_\infty \d \bxi = m\Big\}.$$ 
\end{lemma}
\begin{proof}
    Let $f_1$ and $f_2$ belong to the aforementioned class of solutions and set $\bar{f} := f_1-f_2$. Then, $\int_\Upsilon \bar{f} \d \bxi = 0$, whence it follows that there holds the Poincar\'e inequality $$\Vert \bar{f} \Vert_{L^2(\Upsilon)} \leq C_P \Vert \nabla_{\bxi} \bar{f} \Vert_{L^2(\Upsilon)}.$$
    Furthermore, $\bar{f}$ solves 
    \begin{equation*}
        \Pe \, \dv((1-\rho_\infty) \bar{f} \e(\theta)) = D_e \Delta \bar{f} + \partial^2_\theta \bar{f}, 
    \end{equation*}
    whence testing the equation against $\bar{f}$ yields 
    \begin{equation*}
        D_e \int_\Upsilon |\nabla \bar{f}|^2 \d \bxi + \int_\Upsilon |\partial_\theta \bar{f} |^2 \d \bxi = \Pe \int_\Upsilon (1-\rho_\infty) \bar{f} \e(\theta) \cdot \nabla \bar{f} \d \bxi. 
    \end{equation*}
The Poincar\'e inequality yields 
    \begin{equation*}
        \min\{D_e,1\}\Vert \nabla_{\bxi} \bar{f} \Vert^2_{L^2(\Upsilon)} \leq \Pe \Vert \bar{f} \Vert_{L^2(\Upsilon)} \Vert \nabla \bar{f} \Vert_{L^2(\Upsilon)} \leq \Pe \, C_P \Vert \nabla_{\bxi} \bar{f} \Vert^2_{L^2(\Upsilon)}. 
    \end{equation*}
    We deduce that $\nabla_{\bxi} \bar{f} \equiv 0$ a.e., whence $f_1$ and $f_2$ differ by a constant; the mass constraint $\int_\Upsilon f_1 \d \bxi = \int_\Upsilon f_2 \d \bxi = m$ implies that this constant is null, which concludes the proof. 
\end{proof}

\subsection{Convergence to Equilibrium}\label{subsec:convergence}

We henceforth concentrate entirely on the convergence to the stationary states in the limit of infinite time. Our first result shows the convergence as $t \to \infty$ of the spatial averages $\int_\Omega f(t,x,\theta) \d x$  to the constant stationary state $\frac{1}{2\pi}\int_\Upsilon f_0 \d x \d \theta = \frac{1}{2\pi}\int_\Omega \rho_0 \d x$, regardless of the P\'eclet number. 

\begin{lemma}[Convergence of spatial averages]
    Let $f_0 \in L^2_\per(\Upsilon)$ be admissible non-negative initial data and let $f$ be the unique solution of \eqref{eq:main eqn}. Then, for all $t \geq 0$, 
    \begin{equation*}
        \left\Vert \int_\Omega \!\! \bigg( f(t)  - \frac{1}{2\pi}\int_0^{2\pi}\!\!\! f_0 \d \theta \bigg) \d x \right\Vert_{L^2([0,2\pi])} \leq e^{-t} \left\Vert \int_\Omega \!\! \bigg( f_0 - \frac{1}{2\pi}\int_0^{2\pi}\!\!\! f_0 \d \theta \bigg) \d x \right\Vert_{L^2([0,2\pi])}. 
    \end{equation*}
\end{lemma}

\begin{proof}
    By integrating the equation \eqref{eq:main eqn} with respect to the space variable, we see that $h(t,\theta) := \int_\Omega f(t,x,\theta) \d x$ satisfies the heat equation 
    \begin{equation*}
        \left\lbrace \begin{aligned}
        &\partial_t h = \partial^2_\theta h, \\ 
        &h|_{t=0} = h_0, 
        \end{aligned} \right. 
    \end{equation*}
    where we denote $h_0 := \int_\Omega f_0(x,\theta) \d x$. Setting $$\bar{h} := h - \frac{1}{2\pi}\int_0^{2\pi} h_0 \d \theta,$$ the conservation of the initial mass implies 
    \begin{equation}\label{eq:zero mass diff heat eqn}\int_0^{2\pi} \bar{h} \d \theta = 0. 
    \end{equation}
    Performing the usual parabolic estimate for the equation satisfied by $\bar{h}$ and using Wirtinger's inequality (\textit{i.e.}~Poincar\'e's inequality in one dimension, for which the constant is $1$), we get 
    \begin{equation*}
        \frac{1}{2}\frac{\der}{\der t} \int_0^{2\pi} |\bar{h}|^2 \d \theta = -\int_0^{2\pi} |\partial_\theta \bar{h}|^2 \d \theta \leq -\int_0^{2\pi} |\bar{h}|^2 \d \theta. 
    \end{equation*}
    Gr\"onwall's Lemma then implies 
    \begin{equation*}
        \int_0^{2\pi} |\bar{h}(t)|^2 \d \theta \leq \bigg(  \int_0^{2\pi} |\bar{h}(0)|^2 \d \theta\bigg) e^{-2t}, 
    \end{equation*}
    and the result follows. 
\end{proof}

We conclude this section with the proof of Theorem \ref{thm:small peclet conv}, which shows that, under the constraint of small P\'eclet number, weak solutions converge to a constant stationary state prescribed by the initial data.

\begin{proof}[Proof of Theorem \ref{thm:small peclet conv}]
    Define $w:= f - \langle f_0 \rangle$, where $\langle f_0 \rangle = \avint_\Upsilon f_0 \d \bxi$. Observe that $w$ solves 
  \begin{equation}\label{eq:diff eqn for diff ii}
    \left\lbrace\begin{aligned}
        &\partial_t w + \Pe ~ \dv\big( (1-\rho)w \e(\theta) \big) = D_e \Delta w + \partial^2_\theta w - \Pe \langle f_0 \rangle ~ \dv\big((1-\rho) \e(\theta) \big), \\ 
        &w(0,\cdot) = f_0 - \langle f_0 \rangle. 
    \end{aligned}\right.
\end{equation}
Testing the above with $w$ and using Young's inequality gives 
\begin{equation*}
    \begin{aligned}
       \frac{\der}{\der t}\int_\Upsilon w^2 \d \bxi + \min\{D_e,1\}\int_\Upsilon|\nabla_{\bxi} w|^2 \d \bxi \leq \, &\,\frac{\Pe^2}{\min\{D_e,1\}}\int_\Upsilon w^2 \d \bxi \\ 
        &+ 2 \int_\Upsilon 
     \Pe \langle f_0 \rangle \dv\big((1-\rho) \e(\theta) \big)
     \, w    \d \bxi, 
     \end{aligned}
     \end{equation*}
     and the final term on the right-hand side may be controlled as 
     \begin{align*}
     \int_\Upsilon 
     \dv\big((1-\rho) \e(\theta) \big)
     \, w    \d \bxi =
     -
     \int_\Upsilon 
      \nabla \rho \cdot \e(\theta)
     \, w    \d \bxi
     &=
     -
     \int_\Upsilon 
      \nabla(\rho-\langle \rho_0 \rangle) \cdot \e(\theta)
     \, w    \d \bxi
     \\&=     
     \int_\Upsilon 
     (\rho-\langle \rho_0 \rangle)
     \e(\theta) \cdot
     \nabla w \d \bxi
    \\
    &\leq
   2\pi  \nm{w}_{L^2(\Upsilon)}
     \nm{\nabla w}_{L^2(\Upsilon)}, 
\end{align*}
where we used Jensen's inequality to get $$\begin{aligned}\int_\Upsilon |\rho - \langle \rho_0 \rangle|^2 \d \bxi &= \int_\Upsilon \Big| \int_0^{2\pi} \big( f(t,x,\theta') - \langle f_0 \rangle \big) \d \theta' \Big|^2 \d \bxi \\ &\leq 2\pi \int_\Upsilon \int_0^{2\pi} \big| f(t,x,\theta') - \langle f_0 \rangle  \big|^2 \d \theta' \d \bxi \\ &= (2\pi)^2 \Vert w \Vert^2_{L^2(\Upsilon)}.
\end{aligned}$$
It follows that 
\begin{equation*}
    \begin{aligned}
        \frac{\der}{\der t}\Vert w(t,\cdot) \Vert_{L^2(\Upsilon)}^2 \d \bxi + \frac{1}{2}\min\{D_e,1\} \Vert \nabla_{\bxi} w(t,\cdot) \Vert_{L^2(\Upsilon)}^2 \leq \frac{(2\pi)^2\Pe^2(1+\langle f_0\rangle)^2}{\min\{D_e,1\}}\Vert w(t,\cdot)\Vert_{L^2(\Upsilon)}^2, 
    \end{aligned}
\end{equation*}
and, since $\int_\Upsilon w(t,\cdot) \d \bxi = \int_\Upsilon w(0,\cdot) \d \bxi = 0$ for all $t \geq 0$, an application of Poincar\'e's inequality yields 
\begin{equation*}
    \begin{aligned}
        \frac{\der}{\der t}\Vert w(t,\cdot) \Vert_{L^2(\Upsilon)}^2 \d \bxi \leq -\Big(\underbrace{\frac{1}{2}C_P^{-2}\min\{D_e,1\} - \frac{(2\pi)^2\Pe^2(1+\langle f_0\rangle)^2}{\min\{D_e,1\}}  }_{= 2\kappa }\Big)\Vert w(t,\cdot)\Vert_{L^2(\Upsilon)}^2, 
    \end{aligned}
\end{equation*}
where $C_P$ is the Poincar\'e constant on $\Upsilon$. Provided $\kappa$ is positive, which imposes the smallness condition on $\Pe$, Gr\"onwall's Lemma implies 
\begin{equation*}
    \Vert w(t,\cdot) \Vert^2_{L^2(\Upsilon)} \leq \Vert w(0,\cdot) \Vert_{L^2(\Upsilon)}^2e^{-2\kappa t}, 
\end{equation*}
and the proof is complete. 
\end{proof}

\appendix

\section{Proofs of Technical Lemmas}\label{sec:appendix}

\subsection{Proof of Lemma \ref{lem:CZ periodic}
}\label{app:CZ periodic}

We explain the simple idea underlying the proof of Lemma \ref{lem:CZ periodic}. To begin with, the classical Calder\'on--Zygmund Theorem requires the function under consideration to be compactly supported. In order to bypass this, we use the fact that the periodic cell $\Upsilon = (0,2\pi)^3$ may be interpreted as an interior set of the larger set $\Upsilon' := (-2\pi,4\pi)^3$, say. Periodicity then implies that, for all $p \in [1,\infty]$ and $v \in W^{2,p}_\per(\Upsilon)$, there holds 
\begin{equation}\label{eq:small cell to big cell}
    \Vert \nabla_{\bxi}^{j} v \Vert_{L^p(\Upsilon')} = 3^{3} \Vert \nabla_{\bxi}^{j} v \Vert_{L^p(\Upsilon)} \quad \text{for } j =0,1,2, 
\end{equation}
so that one may recover the desired estimate after having localised. 

\begin{proof}[Proof of Lemma \ref{lem:CZ periodic}]
    Let $\eta \in C^\infty_c(\Upsilon')$ be a positive bump function such that $0 \leq \eta \leq 1$ and $\eta \equiv 1$ on $\Upsilon$. Note that there exists a positive constant $C=C(\Upsilon)$ such that 
    \begin{equation}\label{eq:test function bounds}
        \Vert \nabla_{\bxi}^{j} \eta \Vert_{L^\infty(\Upsilon')} \leq C \quad \text{for } j=0,1,2. 
    \end{equation}
 
    \smallskip 
   \noindent 1. \textit{For smooth functions}: Fix $v$ smooth and $\Upsilon$-periodic, and observe that $\eta v \in W^{2,p}_0(\Upsilon')$. It then follows from the classical Calder\'on--Zygmund Inequality (\textit{cf.}~\cite[Corollary 9.10 \S 9.4]{GilbargTrudinger}) that 
   \begin{equation*}
       \Vert \nabla_{\bxi}^2 (\eta v) \Vert_{L^p(\Upsilon')} \leq C\Vert \Delta_{\bxi} (\eta v) \Vert_{L^p(\Upsilon')}. 
   \end{equation*}
   Using the product rule and the triangle inequality, we also have 
   \begin{equation*}
    \Vert \eta \nabla^2_{\bxi} v \Vert_{L^p(\Upsilon')} \leq \Vert \nabla^2_{\bxi}(\eta v ) \Vert_{L^p(\Upsilon')} +  \Vert \nabla_{\bxi} v \otimes \nabla_{\bxi} \eta +  \nabla_{\bxi} \eta \otimes \nabla_{\bxi} v + v \nabla_{\bxi}^2 \eta \Vert_{L^p(\Upsilon')}, 
   \end{equation*}
whence combining with the previous estimate and \eqref{eq:test function bounds} yields 
   \begin{equation}\label{eq:almost there periodic CZ}
    \Vert \eta \nabla^2_{\bxi} v \Vert_{L^p(\Upsilon')} \leq C\Big( \Vert \Delta_{\bxi}(\eta v ) \Vert_{L^p(\Upsilon')} +  \Vert \nabla_{\bxi} v \Vert_{L^p(\Upsilon')} +  \Vert v\Vert_{L^p(\Upsilon')} \Big). 
   \end{equation}
The inclusion $\Upsilon \subset \Upsilon'$ and the condition $\eta \equiv 1$ on $\Upsilon$ implies that the left-hand side of the previous estimate is bounded from below by  $\Vert \nabla^2_{\bxi} v \Vert_{L^p(\Upsilon)}$. In turn, using \eqref{eq:small cell to big cell} and the periodicity of $v$ to control the final two terms on the right-hand side of \eqref{eq:almost there periodic CZ}, we obtain 
   \begin{equation*}
    \Vert \nabla^2_{\bxi} v \Vert_{L^p(\Upsilon)} \leq C\Big( \Vert \Delta_{\bxi}(\eta v ) \Vert_{L^p(\Upsilon')} +  \Vert v\Vert_{W^{1,p}(\Upsilon)} \Big). 
   \end{equation*}
By expanding 
\begin{equation*}
    \Delta_{\bxi}(\eta v) = v \Delta_{\bxi} \eta + 2 \nabla_{\bxi} \eta \cdot \nabla_{\bxi} v + \eta \Delta_{\bxi} v, 
\end{equation*}
and using again \eqref{eq:small cell to big cell} and \eqref{eq:test function bounds}, we estimate the first term on the right-hand side of \eqref{eq:almost there periodic CZ} to deduce 
   \begin{equation}\label{eq:first step proof of CZ periodic}
    \Vert \nabla^2_{\bxi} v \Vert_{L^p(\Upsilon)} \leq C\Big( \Vert \Delta_{\bxi} v \Vert_{L^p(\Upsilon)} +  \Vert v\Vert_{W^{1,p}(\Upsilon)} \Big) 
   \end{equation}
   for all smooth $\Upsilon$-periodic $v$.

\smallskip

 \noindent 2. \textit{Density argument}:  Now fix $v \in W^{1,p}_\per(\Upsilon)$ satisfying $\Delta_{\bxi} v \in L^p(\Upsilon)$ as per the statement. Let $\{\eta_\varepsilon\}_\varepsilon$ be the usual sequence of Friedrichs mollifiers on $\mathbb{R}^3$; note that the convolutions $$v * \eta_\varepsilon(\bxi) = \int_{\mathbb{R}^3} \eta_\varepsilon(\bxi-\bzeta) v(\bzeta) \d \bzeta$$ are well-defined by virtue of $v$ being locally integrable on any subset of $\mathbb{R}^3$ due to its periodicity. It follows from \eqref{eq:first step proof of CZ periodic} that, since $v*\eta_\varepsilon$ is smooth and periodic, there holds 
  \begin{equation*}
   \begin{aligned}
       \Vert \nabla^2_{\bxi} v * \eta_\varepsilon \Vert_{L^p(\Upsilon)} &\leq C\Big( \Vert \Delta_{\bxi} v * \eta_\varepsilon \Vert_{L^p(\Upsilon)} +  \Vert v*\eta_\varepsilon\Vert_{W^{1,p}(\Upsilon)} \Big) \\ 
       &\leq C\Big( \Vert \Delta_{\bxi} v \Vert_{L^p(\Upsilon)} +  \Vert v \Vert_{W^{1,p}(\Upsilon)} \Big), 
   \end{aligned} 
     \end{equation*}
     where the second line follows from elementary results on mollifiers. The conclusion of the lemma for general $p \in (1,\infty)$ now follows from letting $\varepsilon \to 0$. 

     \smallskip 

     \noindent 3. \textit{Case $p=2$}: In the specific case $p=2$, we integrate by parts twice, noting that the periodicity of $v$ implies that all boundary terms cancel, to obtain 
     \begin{equation*}
         \begin{aligned}
             \Vert \nabla_{\bxi}^2 v \Vert^2_{L^2(\Upsilon)} = \sum_{ij} \int_{\Upsilon} \partial_{ij} v \partial_{ij} v \d \bxi = \sum_{ij} \int_\Upsilon \partial_{ii}v \partial_{jj} v \d \bxi 
             = \int_\Upsilon \Big( \sum_i  \partial_{ii}v \Big) \Big( \sum_j \partial_{jj} v \Big) \d \bxi, 
         \end{aligned}
     \end{equation*}
    and hence $\Vert \nabla_{\bxi}^2 v \Vert^2_{L^2(\Upsilon)} = \Vert \Delta_{\bxi} v \Vert^2_{L^2(\Upsilon)}$, which concludes the proof. 
\end{proof}

\subsection{Proof of Lemma \ref{lem:equiv weak}}\label{app:alternative weak formulation}

\begin{proof}[Proof of Lemma \ref{lem:equiv weak}]

Let $f$ be the unique weak solution of \eqref{eq:main eqn}. Recall from \cite[\S 3]{bbes} that $f$ was constructed via a Galerkin approximation, and may be written as a subsequential limit of the sequence of smooth functions $\{f_n\}_n$ which solve, as a pointwise equality between continuous functions, 
\begin{equation}\label{eq:galerkin approx}
        \partial_t f_n + \dv((1-(\rho_n)_+)_+ f_n \e(\theta)) = \Delta_{\bxi} f_n, 
        \end{equation}
        where $\rho_n = \int_0^{2\pi} f_n \d \theta$ is also smooth, with initial data $f_n(0,\cdot) = f_{0,n} \in L^2_\per(\Upsilon)$ a smooth approximation of $f_0$ in terms of the Galerkin basis functions; \textit{i.e.}~$\lim_{n\to\infty}\Vert f_0 - f_{0,n} \Vert_{L^2(\Upsilon)}=0$. More precisely, with $f_n$ as above, there holds 
        \begin{equation}
        \lim_{n\to\infty}\Vert f_n - f \Vert_{L^2(0,T;H^1(\Upsilon))}=0, 
        \end{equation}
        and, for each $n$, the function $f_n : [0,T]\times\mathbb{R}^3 \to \mathbb{R}$ is such that, for each fixed $t \in [0,T]$, the function $f_n(t,\cdot)$ is triply $2\pi$-periodic and $C^1$. 
The strong convergence in $L^2(0,T;H^1(\Upsilon))$ and the periodicity implies, using a straightforward covering argument, that for any bounded domain $\omega \subset \mathbb{R}^3$, there holds 
\begin{equation}\label{eq:on any bounded domain fn - f}
    \lim_{n\to\infty}\Vert f_n - f \Vert_{L^2(0,T;H^1(\omega))} = 0. 
\end{equation}
Similarly, $\rho_n(t,\cdot)$ is doubly $2\pi$-periodic for each fixed $t \in [0,T]$, and 
\begin{equation}\label{eq:on any bounded domain rhon - rho}
\lim_{n\to\infty}\Vert \rho_n - \rho \Vert_{L^2(0,T;H^1(\omega'))}=0 
\end{equation}
for all bounded domains $\omega' \subset \mathbb{R}^2$. Using \eqref{eq:on any bounded domain fn - f}, a standard argument using Minkowski's inequality 
implies (up to a subsequence which we do not relabel) 
\begin{equation}\label{eq:on any bounded domain fn - f fixed time}
    \lim_{n\to\infty}\Vert f_n(t,\cdot) - f(t,\cdot) \Vert_{H^1(\omega)} = 0 \quad \text{a.e.~}t \in [0,T]. 
\end{equation}

Let $\varphi$ be as in the statement of Lemma \ref{lem:equiv weak}. By testing against the pointwise equality \eqref{eq:galerkin approx} and integrating by parts, using the compact support of $\varphi(t,\cdot)$ for all $t$, we obtain 
\begin{equation*}
   \begin{aligned} \int_{t_1}^{t_2} \int_{\mathbb{R}^3} f_n \partial_t \varphi \d \bxi \d t + \int_{t_1}^{t_2} \int_{\mathbb{R}^3} (1-(\rho_n)_+)_+ f_n \e(\theta) \cdot \nabla \varphi \d \bxi \d t &- \int_{t_1}^{t_2} \int_{\mathbb{R}^3} \nabla_{\bxi} f_n \cdot \nabla_{\bxi} \varphi \d \bxi \d t \\ 
   &= \int_{\mathbb{R}^3} f_n \varphi \d \bxi \Big|_{t_2}  -  \int_{\mathbb{R}^3} f_n \varphi \d \bxi \Big|_{t_1} 
   \end{aligned}
\end{equation*}
for all $n$. By passing to the limit in $n$, using the strong local convergences of \eqref{eq:on any bounded domain fn - f}-\eqref{eq:on any bounded domain fn - f fixed time} and the compact support of the test function in the variable $\bxi$, we obtain 
\begin{equation*}
      \begin{aligned} \int_{t_1}^{t_2} \int_{\mathbb{R}^3} f \partial_t \varphi \d \bxi \d t + \int_{t_1}^{t_2} \int_{\mathbb{R}^3} (1-(\rho)_+)_+ f \e(\theta) \cdot \nabla \varphi \d \bxi \d t - & \int_{t_1}^{t_2} \int_{\mathbb{R}^3} \nabla_{\bxi} f \cdot \nabla_{\bxi} \varphi \d \bxi \d t \\ 
   &= \int_{\mathbb{R}^3} f \varphi \d \bxi \Big|_{t_2}  -  \int_{\mathbb{R}^3} f \varphi \d \bxi \Big|_{t_1}. 
   \end{aligned}
\end{equation*}
We recall from \cite[\S 3.3]{bbes} that $\rho$ satisfies  the estimates \eqref{eq:rho bounds from bbes}, whence the relation \eqref{eq:equiv weak form} follows immediately. 
\end{proof}

\printbibliography

\end{document}